\documentclass[10pt]{amsart}
\usepackage{amsmath,amssymb,amsthm,amsfonts,mathrsfs,amsopn}
\usepackage[all]{xy}
\usepackage{dsfont}
\usepackage{hyperref}
\usepackage{color}
\usepackage{esint}
\usepackage{mathtools}
\mathtoolsset{showonlyrefs}
\usepackage{slashed}
\usepackage{tikz-cd}
\usepackage{graphicx}

\usepackage{color}

\usepackage[margin=1.1in]{geometry}

\newcommand{\p}{\partial}
\newcommand{\R}{\mathbb{R}}
\newcommand{\C}{\mathbb{C}}

\newcommand{\D}{\slashed{D}}

\newcommand{\dd}{\mathop{}\!\mathrm{d}}
\newcommand{\dv}{\dd{vol}}
\newcommand{\Neh}{\overline{N}} % Nehari manifold
\newcommand{\sph}{\mathbb{S}}

\newcommand{\CM}{{C.M.}} % center of mass
\newcommand{\Abracket}[1]{\left<#1\right>} % angle bracket
\newcommand{\parenthesis}[1]{\left(#1\right)} % round bracket
\newcommand{\braces}[1]{\left\{#1\right\}} % curly bracket

\DeclareMathOperator{\diverg}{div}

\DeclareMathOperator{\SO}{SO}
\DeclareMathOperator{\Spec}{Spec}

\DeclareMathOperator{\Span}{Span}

\DeclareMathOperator{\End}{End}
\DeclareMathOperator{\Aut}{Aut}
\DeclareMathOperator{\Id}{Id}
\DeclareMathOperator{\grad}{grad}
\DeclareMathOperator{\PSL}{PSL}
\DeclareMathOperator{\Eigen}{Eigen}
\DeclareMathOperator{\Hess}{Hess}
\DeclareMathOperator{\Vol}{Vol}
\DeclareMathOperator{\tr}{tr}
\DeclareMathOperator{\EL}{EL}
\DeclareMathOperator{\supp}{supp}

\numberwithin{equation}{section}

\newtheorem{thm}{Theorem}[section]

\newtheorem{lemma}[thm]{Lemma}

\newtheorem{rmk}{Remark}[section]

\title[Super-Liouville equations on the sphere]{Existence results for super-Liouville equations on the sphere via bifurcation theory}

\author[A. Jevnikar]{Aleks Jevnikar}
\address{Aleks Jevnikar, Department of Mathematics, Computer Science and Physics, University of Udine, Via delle Scienze 206, 33100 Udine, Italy.}
\email{aleks.jevnikar@uniud.it}

\author[A. Malchiodi]{Andrea Malchiodi}
\address{Andrea Malchiodi, Scuola Normale Superiore, Piazza dei Cavalieri 7, 56126 Pisa, Italy}
\email{andrea.malchiodi@sns.it}

\author[R. Wu]{Ruijun Wu}
\address{Ruijun Wu, Scuola Normale Superiore, Piazza dei Cavalieri 7, 56126 Pisa, Italy}
\email{ruijun.wu@sns.it}

\begin{document}

\begin{abstract}
We are concerned with super-Liouville equations on $\sph^2$, which have variational structure with a strongly-indefinite functional. We prove the existence of non-trivial solutions by combining the use of Nehari manifolds, balancing conditions and bifurcation theory.
\end{abstract}

\maketitle

\begin{center}
	{\em Dedicated to Alice Chang and Paul Yang with admiration}
\end{center}

\

{\footnotesize

\medskip

}

\section{introduction}

In this paper we study the super-Liouville equations, arising from Liouville field theory in supergravity.
Recall that the classical Liouville field theory describes the matter-induced gravity in dimension two:
the super-Liouville field theory is a supersymmetric generalization of the classical one, by taking the spinorial super-partner into account, so that the bosonic and fermionic fields couple under the supersymmetry principle.
Such models also play a role in superstring theory.
For the physics of the Liouville field theory and super-Liouville field theory as well as their relations, one can refer to~\cite{polyakov1981quantumbosonic,polyakov1981quantumfermionic,polyakov1987contemporary,seiberg1990notes,bilal2011super}, and for the applications of Liouville field theory in other models of mathematical physics~\cite{spruckyang1992moltivortices,tarantello1996multiple,yang2001solitons,teschner2003from} and the references therein.
It is almost impossible to have a complete references for the related theory.
However, the existence theory of regular solutions of the super-Liouville equations on closed Riemann surfaces, especially on the sphere, is still far from satisfactory.

\medspace

Liouville equations also have a relevant role in two-dimensional geometry.
For example, on a Riemannian surface~$(M^2,g)$, the Gaussian curvature~$K$ of a conformal metric~$\widetilde{g}\coloneqq e^{2u} g$, with~$u\in C^\infty(M)$, is given by
\begin{equation}\label{eq:Gaussian curvature of conformal metric}
K_{\widetilde{g}}= e^{-2u}(K_g -\Delta_g u).
\end{equation}
Conversely, we have the \emph{prescribed curvature problem}: which functions~$\widetilde{K}$ can be the Gaussian curvatures of a Riemannian metric conformal to~$g$?
If~$M$ is a closed surface, the problem reduces to solving  equation~\eqref{eq:Gaussian curvature of conformal metric} in~$u$ for~$K_{\tilde{g}} = \widetilde{K}$ assigned.
This question has been widely studied in the last century, and the solvability of~\eqref{eq:Gaussian curvature of conformal metric} depends on the geometry and the topology of the surface.
For a surface with nonzero genus, this can be solved variationally, as long as~$\widetilde{K}$ satisfies some mild constraints, see~\cite{kazdan1975scalar, schoenyau1994lectureOnDG-I, aubin1998somenonlinear}.
However when the genus is zero, namely~$M$ is a topological two-sphere, the problem has additional difficulties arising from the non-compactness of the automorphism group.
Actually, since there is only one conformal structure on~$\sph^2$, we can take without loss of generality the standard round metric~$g=g_0$, which is the one induced  by the embedding~$\sph^2\subset \R^3$ with Gaussian curvature~$K_{g_0}=1$.
Let~$x=(x^1,x^2,x^3)$ be the standard coordinates of~$\R^3$.
It was shown in~\cite{kazdan1975scalar} that a necessary condition for~$\widetilde{K}$ to admit a solution~$u$ of~\eqref{eq:Gaussian curvature of conformal metric} is that
\begin{equation}
\int_{\sph^2} \langle \nabla \widetilde{K}, \nabla x^j \rangle e^{2u}\dv=0,\qquad \forall\quad 1\le j\le 3,
\end{equation}
where the volume form~$\dv$ and the gradient are taken with respect to~$g_0$. The above formula shows that,
for example, affine functions cannot be prescribed conformally as Gaussian curvatures.

One of the first existence results for the problem on the sphere is due to Moser, see \cite{moser73onanonlinear}: he proved that
there exist solutions provided that $\widetilde{K}$ is an antipodally-symmetric function.
Other important results were proven in \cite{changyang1987prescribing,changyang1988conformal},  removing the symmetry condition and replacing it with an index-counting condition or some assumption of min-max type, see also \cite{chending1987scalar}.
One fundamental tool in proving such results was an improved Moser-Trudinger inequality derived in \cite{aubin1979meilleures}
for functions satisfying a {\em balancing condition}, namely for which the conformal volume has
zero center of mass in $\R^3$ (where $\sph^2$ is embedded). This fact allowed to show that whenever solutions (or approximate solutions)
of \eqref{eq:Gaussian curvature of conformal metric} blow-up, they develop a single-bubbling behavior. With this information
at hand, existence results were derived via asymptotic estimates and Morse-theoretical results.
We should also mention that there are natural generalizations to higher dimensions, see e.g. \cite{malchiodimayer2019prescribing} and references therein.

\

Recently Jost et al. in \cite{jost2007superLiouville} considered a mathematical version of the super-Liouville equations on surfaces.
Given a  Riemann surface~$M$ with metric~$g$, and~$S\to M$  the spinor bundle with Dirac operator~$\D$,
they considered the Euler--Lagrange equations of the functional ~$\widetilde{I}\colon  H^1(M)\times H^{\frac{1}{2}}(S)\to\R$ given by
\begin{equation}
 \widetilde I(u,\psi):=\int_M \biggr( |\nabla u|^2+ 2K_g u-e^{2u} +2\left<(\D+e^u)\psi,\psi\right> \biggr) \dv_g.
\end{equation}
In subsequent works they performed blow-up analysis and studied the compactness of sequences of solutions  under weak assumptions and in various settings; see e.g.~\cite{jost2009energy, jost2014qualitative,jost2015LocalEstimate} and the references therein.

In ~\cite{jevnikar2019existence} we studied the existence issue from a variational viewpoint when~$M$ is a closed surface of genus~$\gamma>1$, with the signs of some terms  adapted to the background geometry.
More precisely  we consider a uniformized surface~$(M,g)$ with~$K_g=-1$ and  the following functional
\begin{equation}
 \widetilde{J}_\rho(u,\psi)\coloneqq \int_M \left( |\nabla u|^2-2u+e^{2u}
            +2\left<(\D-\rho e^u)\psi,\psi\right> \right) \dv_g,
\end{equation}
where $\rho>0$ is a parameter. The pair~$(0,0)$ is clearly a \emph{trivial} critical point of~$\widetilde{J}_\rho$.
Moreover, when~$\rho$ is not in the spectrum of the Dirac operator~$\D$, we could find non-trivial solutions using min-max schemes.
However, the method there does not directly apply to the sphere case, for two reasons.
First, in the sphere case the trivial solution~$(0,0)$ is not isolated, but within a continuum of solutions connecting to it which are geometrically also \emph{trivial} and induced by M\"obius maps.
Second, there is neither local mountain-pass geometry nor local linking geometry in zero genus, preventing us from finding min-max critical points starting from~$(0,0)$.
Thus, the problem in the sphere case is more challenging.

In this article, we use a Morse-theoretical approach combined with bifurcation theory to attack the problem.
Taking the Gauss--Bonnet formula into account
we consider the following functional
\begin{equation}\label{eq:super-Liouville function-sphere case}
 J_\rho(u,\psi)
 =\int_{\sph^2}\left(|\nabla u|^2+2K_g u-e^{2u}+2\left<(\D-\rho e^u)\psi,\psi\right>\right)\dv_g+ 4\pi,
\end{equation}
where~$g$ is a Riemannian metric on~$\sph^2$,~$\dv_g$ is the induced volume form, and the last tail-term~$4\pi=2\pi\chi(\sph^2)$  is simply needed to normalize the functional so that~$J_\rho(0,0)=0$.
The Euler--Lagrange equations for $J_\rho$ are the following
\begin{equation}\label{eq:EL for super-Liouville}
\tag{$\EL$}
 \begin{cases}
  -\Delta_g u{}&={}e^{2u}-K_g+\rho e^u|\psi|^2, \vspace{0.2cm}\\
  \D_g\psi{}&={}\rho e^u\psi.
 \end{cases}
\end{equation}
Let~$u_*$ be a solution of
\begin{equation}
 -\Delta_g u_* = e^{2u_*}-K_g,
\end{equation}
whose existence follows from the uniformization theorem: then we have clearly a \emph{trivial} solution~$(u_*,0)$ of \eqref{eq:EL for super-Liouville}.
However, in contrast to the higher genus case, here we have another explicit family of solutions with nonzero spinor
component and constant function component, see below.
Hence we are interested in finding \emph{non-trivial solutions} with non-constant components.

We remark that for the system~\eqref{eq:EL for super-Liouville} to admit a solution with nonzero spinor component, it is necessary that~$\rho >1$.
Indeed, for every solution ~$(u,\psi)$  (which is smooth by  regularity theory, see~\cite{jost2007superLiouville, jevnikar2019existence})  we can consider the metric~$g_u\coloneqq e^{2u} g$ on~$\sph^2$.
The corresponding Dirac operator~$\D_{g_u}$ has~$\rho$ as an eigenvalue since the second equation transforms into
\begin{equation}
 \D_{g_u}\psi_u= \rho \, \psi_u,
\end{equation}
where~$\psi_u=e^{-\frac{1}{2}u}\beta(\psi)$ for an isometric isomorphism~$\beta\colon S_g\to S_{g_u}$ on the corresponding spinor bundles.
Meanwhile, the first equation implies that the volume of the new metric~$g_u$ satisfies
\begin{equation}
 \Vol(\sph^2,g_u)=\int_{\sph^2} \dv_{g_u}
 =\int_{\sph^2}e^{2u}\dv_{g}
 =\int_{\sph^2} K_g \dv_g -\rho\int_{\sph^2} e^u|\psi|^2\dv_g
 \le 4\pi.
\end{equation}
It is known from~\cite{baer1992lower} that
\begin{equation}
 \lambda_1(\D_{g'})^2\Vol(\sph^2, g') \ge 4\pi
\end{equation}
for any metric~$g'$ on~$\sph^2$.
In particular, we conclude that~$\rho> 1$ if~$\psi$ is not identically zero.

Without loss of generality, we may consider the standard round sphere~$(\sph^2, g_0)$ with~$K_{g_0}\equiv 1$.
This is due to the conformal covariance of the system~\eqref{eq:EL for super-Liouville}, see Section \ref{sect:conformal symmetry}.
Then the trivial solutions are simply~$\theta=(0,0)\in H^1(M)\times H^{\frac{1}{2}}(S)$ and its M\"obius transformations, see again Section \ref{sect:conformal symmetry}.
On the round sphere we know that the eigenspinors corresponding to the eigenvalue~$\lambda_1=1\in \Spec(\D_{g_0})=\mathbb{Z}\backslash \{0\}$ has constant length, i.e. if~$\D\varphi_1=\varphi_1$, then the function~$|\varphi|\colon \sph^2 \to \R$ is constant.
Such spinors constitute a vector space of real dimension 4.
This allows us to construct another family of solutions, namely choosing~$u$ to be the constant function such that~$\rho \, e^u=\lambda_1$ and then choosing~$\psi\in \Eigen(\D_{g_0}, \lambda_1)$ of a length such that the first equation of~\eqref{eq:EL for super-Liouville} holds.
Therefore, for any~$\rho\ge 1$, let~$\varphi_1\in\Eigen(\D_{g_0},1)$ be an eigenspinor of unit length: then the pair
$$
 u=-\ln\rho, \quad \psi=\frac{\sqrt{\rho^2-1}}{\rho}\,\varphi_1
$$
is a solution of~\eqref{eq:EL for super-Liouville}.
Note that these solutions converge to the trivial solution~$\theta=(0,0)$ as~$\rho\to 1$, which highlights a bifurcation phenomenon at the first eigenvalue~$\rho=\lambda_1$.
We will see that this is actually a more general phenomenon.
For later convenience we call a solution~$(u,\psi)$ \emph{non-trivial} if the function component~$u$ is not constant and the
pair $(u,\psi)$ is not in the conformal orbit of constant functions.
Note that~$u=const.$ implies that~$|\psi|=const.$, which is only the case if~$\psi$ is a Killing spinor and~$\rho e^u=1$.
Also, the eigenspinors for~$\lambda_k>1$ do not have constant length, see~\cite[Section 2.2]{camporesi1996ontheeigenfunction} and~\cite[Section 4.2]{hermann2012diraceigenspinors}. 

\begin{thm}\label{thm:existence around eigenvalues}
Let~$\rho=\lambda_k\in\Spec(\D)$ with $\lambda_k> 1$.
Then, $\rho$ is a bifurcation point for \eqref{eq:EL for super-Liouville} on $\sph^2$, i.e. there exists a sequence $\rho_l\to \rho=\lambda_k$ such that \eqref{eq:EL for super-Liouville} admits a non-trivial solution on~$\sph^2$ for $\rho=\rho_l$.
\end{thm}
The metric~$g$ in the above statement is suppressed: once we proved it for the round metric~$g_0$, then it also holds for any other (smooth) metric~$g$ by conformal and diffeomorphism transformations since, as we recalled, the sphere admits only one conformal class of metrics.

Note that there exists a~$3$-dimensional family of quaternionic structures on the spinor bundle~$S$, which are fibrewise automorphisms preserving the connection, metric and Clifford multiplication.
Thus, once we get a solution with nonzero spinor component, we automatically get a three-dimensional family of solutions for free.

There also exists  the real volume element~$\omega=e_1\cdot e_2$, where~$(e_1, e_2)$ denotes a local oriented orthonormal frame of~$\sph^2$ and the dot is  the Clifford multiplication in the Clifford bundle~$Cl(\sph^2)$.
It is readily checked that~$\omega$ is globally well-defined.
The endomorphism~$\gamma(\omega)\equiv \gamma(e_1)\gamma(e_2)\in \End(S)$ is an \emph{almost-complex} structure,
parallel with respect to the spin connection, but anti-commutative with the Dirac operator:~$\D(\gamma(\omega)\psi)=-\gamma(\omega)\D\psi$.
Therefore, if~$(u,\psi)$ is a solution to~\eqref{eq:EL for super-Liouville}, then the pair~$(u,\gamma(\omega)\psi)$ solves the system
\begin{equation}
 \begin{cases}
  -\Delta_g u{}&={}e^{2u}-K_g+\rho e^u|\gamma(\omega)\psi|^2, \vspace{0.2cm}\\
  \D_g(\gamma(\omega)\psi){}&={}-\rho e^u\left(\gamma(\omega)\psi\right).
 \end{cases}
\end{equation}
That is, we can allow a change of  sign in front of the spinorial part in the functional~$J_\rho(u,\psi)$, without affecting the result.

\medskip

The main observation is that the second equation in~\eqref{eq:EL for super-Liouville} has the form of a weighted eigenvalue equation.
This suggests  to employ a bifurcation argument to search for non-trivial solutions.
Recall that a theorem by Krasnosel'skii states that for a pure (nonlinear) eigenvalue problem, any eigenvalue is a bifurcation point for the eigenvalue equation, see e.g. \cite[Chapter 5, Appendix]{ambrosettiprodi1993primer} and~\cite{krasnoselskii1964topological} with the references therein.
Here we are adopting a Morse-theoretical approach in the spirit of~\cite{marinoprodi1968teoria}, see also~\cite[Sect. 12]{ambrosetti2007nonlinear}.
However, note that here the presence of the Dirac operator makes the functional strongly indefinite and the Morse-theoretical groups are generally not well-defined, meanwhile the critical points are not isolated because of the symmetries of the functional.
To overcome these difficulties we introduce some natural constraints, based on spectral decomposition and balancing conditions, to remove most of the negative directions which decreases the functional and also kill the redundancy of the conformal orbits.
Restricted to this \emph{Nehari type manifold}, the origin is now an isolated critical point, and though the functional is still indefinite, we are able to count the index of the origin within the Nehari manifold and hence get the well-defined local critical groups.
In doing so we reduce ourselves to a more classical setting and the problem is tractable: see also \cite{fitzpatrick1999spectral}
for related issues treated via spectral flows.

The paper is organized as follows.
First we recall some preliminary facts about the Dirac operator and set-up the variational framework.
Then we introduce a class of Nehari manifolds and show that they are natural constraints. After showing the
validity of the Palais-Smale condition,
 we analyze the local behavior of the functional around the origin and define the critical groups there.
In the end we use a parametrized flow to show the bifurcation result, hence obtaining the existence of non-trivial solutions.

\

\noindent {\bf Acknowledgments.}
A.M. has been partially supported by the project {\em Geometric problems with loss of compactness}  from Scuola Normale Superiore.
A.J. and A.M. have been partially supported by MIUR Bando PRIN 2015 2015KB9WPT$_{001}$.
They are also members of GNAMPA as part of INdAM.
R.W. is supported by Centro  di  Ricerca  Matematica  \emph{Ennio  de  Giorgi}.

\
%%%%%%%%%%
\section{Preliminaries}

Recall that $\sph^2$ admits a unique conformal structure up to diffeomorphism and consider the  Riemannian metric~$g_0$ induced from the embedding ~$\sph^2\subset\R^3$.
The spectrum of the Laplace operator~$-\Delta_{\sph^2}=-\Delta_{g_0}$  is explicitly known: the eigenvalues are given by~$\mu_k=k(k+1)$, for~$k=0,1,2,\dots$,  the multiplicity of~$\mu_0=0$ is~$1$ (eigenfunctions given by constants), that of~$\mu_1=2$ is~$3$ (with eigenfunctions given by affine functions on~$\R^3$ restricting to~$\sph^2$; a basis is given by the coordinate functions~$\{x^1, x^2, x^3\}$), with  multiplicities  of~$\mu_k (k\ge 2)$  given by the binomial coefficients
\begin{equation}
 \binom{2+k}{2}-\binom{k}{2}
\end{equation}
and with eigenfunctions given by homogeneous harmonic polynomials on~$\R^3$ restricted to~$\sph^2$,
see e.g.~\cite[Chapter 4]{ambrosettimalchiodi2006perturbation}.

The two-sphere admits a non-compact group of conformal automorphisms, which constitutes the M\"obius group~$\Aut(\C\cup\{\infty\})=\PSL(2;\C)$.
In terms of the Riemann sphere~$\C\cup\{\infty\}$, these are the fractional linear transformations, which are nothing but compositions of translations, rotations, dilations and inversions.
Note that with zero spinor components, the functional
\begin{equation}
 J_\rho(u,0)= \int_{\sph^2} |\nabla u|^2 +2u -e^{2u}\dv_{g_0}+4\pi,
\end{equation}
is invariant under the M\"obius group action.
Indeed, each element~$\varphi\in \PSL(2;\C)$ is a conformal diffeomorphism with~$\varphi^* g_0=\det(\dd\varphi)g_0$.
For any~$u\in H^1(\sph^2)$, set
\begin{equation}
  u_\varphi\coloneqq u\circ \varphi+\frac{1}{2}\ln \det(\dd\varphi),
\end{equation}
then it is a classical fact that~$J_\rho(u_\varphi,0)=J_\rho(u,0)$.

\

Consider the spinor bundle~$S\to\sph^2$ associated to the unique spin structure of~$\sph^2$ and let~$\D=\D_{g_0}$ be the  Dirac operator.
For basic material on spin geometry and Dirac operators, one may refer to~\cite{friedrich2000dirac, ginoux2009dirac, jost2011riemannian, lawson1989spin}.
Recall that the spectrum of the Dirac operator is
\begin{equation}
 \pm(k+1); \qquad k\in\mathbb{N},
\end{equation}
and the eigenvalue~$\pm(k+1)$ has (real) multiplicity~$4(k+1)$.
In particular, there are no harmonic spinors on~$\sph^2$, and the first positive eigenvalue is~$1$ with eigenspinors having constant length (they are actually given by the Killing spinors).
For more details we refer to~\cite[Chapter 2 and Appendix]{ginoux2009dirac}.

\

We give a brief description of the Sobolev spaces~$H^1(\sph)$ and~$H^{\frac{1}{2}}(S)$ which we will employ.
For basic material on Sobolev spaces and fractional Sobolev spaces, see~\cite{gilbarg2001elliptic, ammann2003habilitation}.
Most recent papers on analysis of Dirac operators contain such an introductory part, and here we only collect some necessary material.

The Sobolev space~$H^1(\sph^2)$ is equipped with the inner product
\begin{equation}
 \left<u,v\right>_{H^1}
 = \int_{\sph^2} \left<\nabla u, \nabla v\right>+ u\cdot v \dv.
\end{equation}
For smooth functions (which are dense in~$H^1(\sph^2)$), an integration by parts gives
\begin{equation}
 \left<u,v\right>_{H^1}
 =\int_{\sph^2}[(1-\Delta)u]\cdot v\dv
 =\left<(1-\Delta)u, v\right>_{L^2}
 =\left<(1-\Delta)u, v\right>_{H^{-1}\times H^1},
\end{equation}
where the last bracket denotes the dual pairing.
Note that, in contrast to the case in~\cite{jevnikar2019existence}, here the functional~$u\mapsto J_\rho(\cdot,0)$ is not coercive.
At any~$u\in H^1(\sph^2)$ there are finitely-many negative directions of the Hessian~$\Hess_{u} J_\rho(u,0)$.
Moreover, the functional~$J_\rho(\cdot,0)$ does not admit local linking geometry around the trivial critical point~$u=0$.

The fractional Sobolev space of the sections of the spinor bundle~$S$ can be defined via the~$L^2$-spectral decomposition.
Recall that~$\D$ is a first-order elliptic operator which is essentially self-adjoint and has no kernel:
counting eigenvalues with multiplicities, the eigenvalues~$\{\lambda_k\}_{k\in\mathbb{Z}_*}$ (where~$\mathbb{Z}_*\equiv\mathbb{Z}\backslash\{0\}$) are listed in a non-decreasing order:
\begin{equation}
 -\infty \leftarrow\cdots\le \lambda_{-l-1}\le \lambda_{-l}\le\cdots\le \lambda_{-1}< 0
 < \lambda_1\le \cdots \le \lambda_k \le \lambda_{k+1}\le \cdots \to +\infty.
\end{equation}
Moreover, the spectrum is symmetric with respect the the origin.
Let~$(\varphi_k)_k$ be the eigenspinors corresponding to~$\lambda_k$, ~$k\in\mathbb{Z}_*$ with~$\|\varphi_k\|_{L^2(M)}=1$:
they form a complete orthonormal basis of~$L^2(S)$.
For any spinor~$\psi\in\Gamma(S)$, we have
$$\label{eq:spinor in basis}
 \psi=\sum_{k\in \mathbb{Z}_*}a_k\varphi_k,
\quad
 \D\psi= \sum_{k\in\mathbb{Z}_*} \lambda_k a_k \varphi_k.
$$
For any~$s>0$, the operator~$|\D|^s\colon \Gamma(S)\to \Gamma(S)$ is defined as
\begin{equation}
 |\D|^s\psi =\sum_{k\in\mathbb{Z}_*} |\lambda_k|^s a_k\varphi_k.
\end{equation}
The domain of~$|\D|^s$ is given by the spinors such that   the right-hand side belongs to~$L^2(S)$, i.e.
\begin{equation}
 H^s(S)\coloneqq \left\{\psi\in L^2(S)\mid \int_{\sph^2}\left<|\D|^s\psi,|\D|^s\psi\right>\dv_{g_0} <\infty  \right\},
\end{equation}
which is a Hilbert space with inner product
\begin{equation}
 \left<\psi,\phi\right>_{H^s}
 = \left<\psi,\phi\right>_{L^2}
 +\left<|\D|^s\psi, |\D|^s\phi\right>_{L^2}.
\end{equation}
For~$s=k\in\mathbb{N}$,~$H^k(S)=W^{k,2}(S)$ and the above norm is equivalent to the Sobolev~$W^{k,2}$-norm.
For~$s<0$,~$H^{s}(S)$ is by definition the dual space of~$H^{-s}(S)$.

Since~$S$ has finite rank, the general theory for Sobolev's embedding on closed manifold continues to hold here.
In particular, for~$0<s<1$ and~$q\le \frac{2}{1-s}$, we have the continuous embedding
\begin{equation}
 H^s(S) \hookrightarrow L^{q}(S).
\end{equation}
Furthermore, for~$q<\frac{2}{1-s}$ the embedding is compact,
 see e.g. \cite{ammann2003habilitation} for more details.

Let us now consider the case~$s=\frac{1}{2}$.
Note that for~$\psi\in H^{\frac{1}{2}}(S)$ we have~$\D\psi\in H^{-\frac{1}{2}}(S)$, which is defined in the distributional sense.
Since~$\D$ has no kernel, we can split the spectrum into the positive and negative parts, and accordingly we have the decomposition
\begin{equation}\label{eq:split}
 H^{\frac{1}{2}}(S)= H^{\frac{1}{2},+}(S)\oplus H^{\frac{1}{2},-}(S).
\end{equation}
Let~$\psi=\psi^+ +\psi^-$ with~$\psi^{\pm}\in H^{\frac{1}{2},\pm}(S)$. Then
\begin{align}
 \int_{\sph^2}\Abracket{|\D|\psi,\psi}\dv_{g_0}
 =&
 \int_{\sph^2} \Abracket{|\D|\psi,\psi}\dv_{g_0}\\
 =&
 \int_{\sph^2} \left[ \left<\D\psi^+,\psi^+\right>-\left<\D\psi^-,\psi^-\right> \right] \dv_{g_0}\\
 =&\||\D|^{\frac{1}{2}}\psi^+\|^2_{L^2}
 +\||\D|^{\frac{1}{2}}\psi^-\|^2_{L^2},
\end{align}
the square root of which defines a norm equivalent to the~$H^{\frac{1}{2}}$-norm.
\

\section{Conformal Symmetry}
\label{sect:conformal symmetry}

We next discuss  the conformal symmetries of the functional and of the equations:
these were treated for example in~\cite{jevnikar2019existence}, but
we recall them here for completeness.
Suppose that~$(u,\psi)$ is a solution of~\eqref{eq:EL for super-Liouville},
let~$v\in C^\infty(M)$ and consider the metric~$g_v\coloneqq e^{2v}g$.
There exists an isometric isomorphism~$\beta\colon S_g\to \widetilde{S}_{g_v}$ of the spinor bundles corresponding to different metrics such that
\begin{equation}\label{eq:conformal transform of Dirac operator}
 \D_{g_v}\left(e^{-\frac{v}{2}}\beta(\psi)\right)
 =e^{-\frac{3}{2}v}\beta(\D_g\psi),
\end{equation}
see e.g.~\cite{ginoux2009dirac, hitchin1974harmonicspinors}, where we are using the notation from~\cite{jost2018symmetries}.
Thus the pair
\begin{equation}
 \begin{cases}
  \widetilde{u}=u-v, \vspace{0.2cm}\\
  \widetilde{\psi}= e^{-\frac{u}{2}}\beta(\psi),
 \end{cases}
\end{equation}
solves the system
\begin{align}
 -\Delta_{g_v}\widetilde{u}
 ={}&-e^{-2v}\Delta_g(u-v)= e^{-2v}(e^{2u}-K_g+\rho e^{u}|\psi|^2+\Delta_g v) \\
 ={}&e^{2(u-v)}-e^{-2v}(K_g-\Delta_g v)+\rho e^{u-v}|e^{-\frac{v}{2}} \beta(\psi)|^2 \\
 ={}&e^{2\widetilde{u}}-{K}_{g_v}
 +\rho e^{\widetilde{u}}|\widetilde{\psi}|^2, \\
\widetilde{\D}_{g_v} \widetilde{\psi}
 ={}&\rho e^{-\frac{3}{2}v}\beta(e^u\psi)
     =\rho e^{u-v}\left(e^{-\frac{1}{2}v}\beta(\psi)\right)
     =\rho e^{\widetilde{u}}\widetilde{\psi},
\end{align}
which has the same form as \eqref{eq:EL for super-Liouville}.

The automorphisms group of the Riemann sphere~$\sph^2=\C\cup\{\infty\}$ is a family of conformal
maps that induce a natural action on Sobolev spaces of functions and spinors.
Let~$\varphi\in \PSL(2,\C)=\Aut(\sph^2)$ be a conformal diffeomorphism with~$\varphi^* g_0=\det(\dd\varphi)g_0$.
For any~$(u,\psi)$, we set
\begin{equation}
 \begin{cases}
  u_\varphi\coloneqq u\circ \varphi+\frac{1}{2}\ln \det(\dd\varphi), \vspace{0.2cm} \\
  \psi_\varphi\coloneqq \left(\det(\dd\varphi)\right)^{1/4}\beta(\psi\circ\varphi),
 \end{cases}
\end{equation}
where~$\beta\colon S\to \varphi^*S$ denotes the isometry of the spinor bundles.
Then, not only~$(u_\varphi,\psi_\varphi)$ satisfies~\eqref{eq:EL for super-Liouville}, but also the functional on~$(\sph^2, g_0)$ stays invariant
\begin{equation}
 J_\rho(u_\varphi,\psi_\varphi)=J_\rho(u,\psi).
\end{equation}
This generalizes~\cite[Prop. 2.1]{changyang1987prescribing} in the classical Liouville case.

As consequences of such symmetries, on one hand, for any given metric on the sphere~$\sph^2$, we can use a conformal diffeomorphism to reduce the problem to the case where the metric on~$\sph^2$ is the standard round metric~$g_0$ with~$K_{g_0}\equiv 1$; on the other hand, a critical point~$(u,\psi)$ of~$J_\rho$ is never isolated in~$H^1(\sph^2)\times H^{\frac{1}{2}}(S)$.
Since the elements in the orbits of the conformal transformations are geometrically the same, we will overcome this problem by picking those elements with centers of mass at the origin.

\section{A natural constraint}
\label{section:natural constraint}

Due to the above conformal symmetry,  without loss of generality we may consider the problem with respect to the standard round metric~$g=g_0$.
Then the functional becomes
\begin{equation}\label{eq:functional-uniformized metric}
 J_\rho(u,\psi)
 =\int_M \biggr( |\nabla u|^2+2u-e^{2u}+2\left(\left<\D\psi,\psi\right>-\rho e^u|\psi|^2\right) \biggr)\dv+4\pi,
\end{equation}
whose Euler--Lagrange equations take the following simple form
\begin{equation}\label{eq:EL-uniformized metric}\tag{$\EL_0$}
 \begin{cases}
  -\Delta_g u=e^{2u}-1+\rho e^u|\psi|^2, \vspace{0.2cm}\\
  \D_g\psi= \rho e^u\psi.
 \end{cases}
\end{equation}
In the functional~$J_\rho$, the part involving spinors is strongly indefinite while the
remaining terms are invariant with respect to the M\"obius group: both these properties make the variational approach quite challenging.
We therefore need to confine such defects.

For~$u\in H^1(\sph^2)$, the function~$e^{2u}$ can be considered as a \emph{mass distribution} on~$\sph^2$, see~\cite{changyang1987prescribing}.
Let~$\vec{x}=(x^1, x^2,x^3)\in \R^3$ be the position vector.
The \emph{center of mass} of~$e^{2u}$ is defined as
\begin{equation}
 \CM(e^{2u})\coloneqq \frac{\int_{\sph^2} \vec{x} e^{2u}\dv}{\int_{\sph^2}e^{2u}\dv} \in \R^3.
\end{equation}
For any~$u\in H^1(M)$, there exists a~$\varphi\in \PSL(2,\C)$ such that~$\CM(e^{2u_\varphi})=0\in \R^3$; moreover, the M\"obius transformation can be chosen to depend on~$u$ in a continuous way~\cite[Lemma 4.2]{changyang1987prescribing}.
Note that such a~$\varphi$ is never unique: there is always the freedom of a~$\SO(3)$-action which leaves~$|\CM(e^{2u})|$ invariant.
See~\cite{changyang1987prescribing} for the argument and more information on the center of mass.
We remark that~$\CM(e^{2u})=0$ means that the function~$e^{2u}$ is orthogonal to the first eigenfunctions on~$\sph^2$ with respect to the~$L^2$-inner product.
Let
\begin{equation}
 \widetilde{H}^1(\sph^2)\coloneqq
 \braces{u\in H^1(M): \CM(e^{2u})=0}.
\end{equation}
\begin{lemma}
 $\widetilde{H}^1(\sph^2)$ is a submanifold of~$H^1(\sph^2)$.
\end{lemma}
    \begin{proof}
     Consider the map~$G_1\colon H^1(\sph^2)\to \R^3$ defined by
     \begin{equation}
      G_1(u)=\int_{\sph^2} \vec{x} e^{2u}\dv
      =\parenthesis{\int_{\sph^2} x^1 e^{2u}\dv,\int_{\sph^2} x^2 e^{2u}\dv,\int_{\sph^2} x^3 e^{2u}\dv} \in\R^3.
     \end{equation}
     It suffices to show that~$\dd G_1(u)$ is surjective for each~$u\in H^1(\sph^2)$ and then the preimage~$\widetilde{H}^1(\sph^2)=G_1^{-1}(0)$ is a submanifold.

     The differential is explicitly given by
      \begin{align}
       \dd G_1(u)[v] &= \int_{\sph^2} \vec{x} e^{2u} (2v)\dv\in\R^3.
      \end{align}
     Consider an affine function~$v$ of the form
     \begin{equation}
        v(x)=\sum_{j=1}^3 v_j x^j, \qquad v_j\in \R, \quad j=1,2,3.
     \end{equation}
     For any~$\vec{y}\in\R^3$, we need to solve
     \begin{equation}
        \sum_{j=1}^3 \left(\int_{\sph^2} x^k e^{2u} x^j \dv(x)\right) (2v_j)=y_k, \qquad k=1,2,3.
     \end{equation}
     In~\cite[Section 4]{changyang1987prescribing} it has been shown that the matrix~$\Lambda(u)=(\Lambda_{kj}(u))$, with
     \begin{equation}
     \Lambda_{kj}(u)=\int_{\sph^2}  e^{2u}x^k x^j \dv(x),
     \end{equation}
     is invertible.
     Thus there exists a unique affine function~$v=\sum_{j=1}^3 v_j x^j$, which is clearly in~$H^1(\sph^2)$, such that
     \begin{equation}
      \dd G_1(u)[v]=\vec{y}\in \R^3.
     \end{equation}
     That is,~$\dd G_1(u)$ is surjective to~$\R^3$.
    \end{proof}

%{\color{blue} I WOULD AVOID THE METRIC $G_U$ AND $e^{-u} DIRAC$}

Next, we consider  weighted eigenvalues for the Dirac operator.
Given~$u\in H^1(\sph^2)$, consider the operator~$e^{-u}\D_g$ and write~$\{\lambda_j(u)\}$ and~$\{\varphi_j(u)\}$ for the associated eigenvalues and eigenspinors respectively:
\begin{equation}
  e^{-u}\D_g\varphi_j(u)=\lambda_j(u)\varphi_j(u),
  \qquad \forall j\in\mathbb{Z}_*.
\end{equation}
Since $\D_g$ has no kernel, $e^{-u} \D_g$ also does not: the above equalities could equivalently be viewed as weighted eigenvalue equations
\begin{equation}
 \D_g \varphi_j(u)=\lambda_j(u) e^u\varphi_j(u),
 \qquad \forall j\in \mathbb{Z}_*.
\end{equation}
Furthermore these eigenspinors can be chosen to be orthonormal with respect to the weight~$e^u$, namely, for any~$j,k\in\mathbb{Z}_*$,
\begin{equation}
 \int_{\sph^2} \Abracket{\varphi_j(u),\varphi_k(u)}e^u\dv=\delta_{jk}.
\end{equation}
\begin{rmk}
  If $u$ is a smooth function, we have a conformal metric~$g_u= e^{2u}g$ with~$\dv_{g_u}=e^{2u}\dv_g$.
  Moreover, writing~$\beta\colon S_g\to S_{g_u}$ for the isometric isomorphism of corresponding spinor bundles and setting
  \begin{align}
   (\varphi_j)_u\coloneqq e^{-\frac{u}{2}}\beta(\varphi_j(u)),\qquad \forall j\in \mathbb{Z}, \; j\neq 0,
  \end{align}
 the above formulas are to say that
$$
   \D_{g_u}(\varphi_j)_u=\lambda_j(u)(\varphi_j)_u, \quad
    \int_{\sph^2}\Abracket{(\varphi_j)_u,(\varphi_k)_u}e^u\dv_{g_u}=\delta_{jk}.
$$
\end{rmk}

Note that the operator~$e^{-u}\D$ has analytic dependence in~$u$, thus the weighted eigenvalues~$\parenthesis{\lambda_j(u)}$ and eigenspinors~$\parenthesis{\varphi_j(u)}$ have at least~$C^1$-dependence on~$u$, see e.g.~\cite[Chap. 8, Sect. 2]{kato2013perturbation}.
Fixing now~$u\in H^1(\sph^2)$, we consider the vector space
\begin{align}
 N(u)\coloneqq \braces{\psi\in H^{\frac{1}{2}}(S_g): G_{2,j}(\psi)\equiv\int_{\sph^2}\Abracket{\D\psi-\rho e^{u}\psi,\varphi_j(u)}\dv_g =0,\; \forall j<0}.
\end{align}
Since
\begin{align}
 0=\int_{\sph^2} \Abracket{\D_g\psi-\rho e^u\psi,\varphi_j(u)}\dv_g
 =& \int_{\sph^2}\Abracket{\psi,\D_g\varphi_j(u)}\dv_g
    -\int_{\sph^2} e^u\Abracket{\psi, \varphi_j(u)}\dv_g\\
 =&\; (\lambda_j(u)-\rho)\int_{\sph^2}\Abracket{\psi,\varphi_j(u)}e^u\dv_g
\end{align}
and~$\lambda_j(u)<0$ for~$j<0$ while~$\rho\ge1$, we have
\begin{equation}
 \int_{\sph^2}\Abracket{\psi,\varphi_j(u)}e^u\dv_g=0.
\end{equation}
Thus~$N(u)$ is the set of spinors associated to the positive spectrum of~$e^{-u}\D$:
\begin{align}
 N(u)
 =&\braces{\psi\in H^{\frac{1}{2}}(S): P_{u}^-(\psi)=0},
\end{align}
where~$P_{u}^-\colon H^{\frac{1}{2}}(S)\to H^{\frac{1}{2}}(S)$ denotes the projection to the subspace spanned by the weighted eigenspinors~$\braces{\varphi_j(u): j>0}$.
Note that for~$j<0$ (hence~$\lambda_j(u)<0$)
\begin{equation}
 \int_{\sph^2} \Abracket{\D_g\psi,\varphi_j(u)}\dv_g
 =\int_{\sph^2}\Abracket{\psi,\D_g\varphi_j(u)}\dv_g
 =\lambda_j(u)\int_{\sph^2}\Abracket{\psi,\varphi_j(u)}e^u\dv_g=0.
\end{equation}
For another~$v\in H^1(\sph^2)$ small in norm, the spaces~$N(u+v)$ and~$N(u)$ are isomorphic, by the continuous dependence of the eigenspinors on the weight function.
Define
\begin{equation}
 N\coloneqq\braces{(u,\psi)\in H^1(\sph^2)\times H^{\frac{1}{2}}(S):  G_1(u)=\vec{0}\in\R^3,\; \psi\in N(u)} \subset \widetilde{H}^1(\sph^2)\times H^{\frac{1}{2}}(S).
\end{equation}
Then~$N$ is the total space of the trivial vector bundle ~$\pi\colon N\to\widetilde{H}^1(\sph^2)$ with fiber space~$\pi^{-1}(u)=N(u)$.
In particular,~$N$ is a Hilbert submanifold with  Hilbertian structure induced from the space~$H^1(\sph^2)\times H^{\frac{1}{2}}(S)$.

Now let
\begin{align}
 \Neh_\rho \coloneqq
 \braces{(u,\psi)\in N: \fint_{\sph^2} (e^{2u}+\rho e^u|\psi|^2) \dv_g=1}.
\end{align}
To see that it is a sub-manifold, we consider the following map
\begin{align}
 G_3\colon N\to\R, & &
 G_3(u,\psi)\coloneqq \fint_{\sph^2}( e^{2u}+\rho e^u|\psi|^2-1)\dv_g.
\end{align}
Its differential is given by
\begin{equation}
 \dd G_3(u,\psi)[v,h]
 =\fint_{\sph^2}(2v e^{2u}+\rho v e^v|\psi|^2 +2\rho e^v\Abracket{\psi,h})\dv_g
\end{equation}
for~$(v,h)\in T_{u,\psi}N$.
For any~$t\in\R$, if~$\psi=0$ then one can find~$v\in T_u \widetilde{H}^1(\sph^2)$ such that~$\dd G_3(u,0)[v,0]=t$; otherwise~$\psi\neq 0$, we can take~$v=0$ and~$h=s\psi$ for some~$s\in\R$ such that~$\dd G_3(u,\psi)[0,s\psi]=t$.
Thus~$\dd G_3(u,\psi)\colon T_{(u,\psi)}N\to \R$ is always surjective and the preimage~$\Neh_\rho=G_3^{-1}(0)\subset N$ is a submanifold.

To summarize, the subset
\begin{align}
 \Neh_\rho=
 \braces{(u,\psi)\in H^1(\sph^2)\times H^{\frac{1}{2}}(S):
    G_1(u)=0, G_{2,j}(u,\psi)=0, (\forall j<0), \; G_3(u,\psi)=0
 }
\end{align}
is a connected infinite-dimensional manifold with an induced Hilbertian structure.
Restricting the functional~$J_\rho$ to this submanifold~$J_\rho|_{\Neh_\rho}\colon \Neh_\rho\to\R$, we consider the constrained critical points~$(u,\psi)\in \Neh_\rho$ which satisfy the constrained Euler--Lagrange equations
\begin{align}\label{eq:constrained EL-u}
 -\Delta & u+1- e^{2u}-\rho e^u|\psi|^2\\
 =&\sum_{j=1}^3 \alpha_j x^j e^{2u}
   +2\sum_{k<0}\mu_k\parenthesis{\Abracket{\D\psi-\rho e^u\psi,\delta_u\varphi_k(u)}-\rho e^u\Abracket{\psi,\varphi_j(u)}}
 +2\tau\parenthesis{2e^{2u}+\rho e^u|\psi|^2}
\end{align}
\begin{align}\label{eq:constrained EL-psi}
 \D\psi-\rho e^u\psi
 =&\sum_{k<0}\mu_k \parenthesis{\D\varphi_k(u)-\rho e^u\varphi_k(u)}
    +\tau \rho e^u\psi,
\end{align}
where~$\alpha_j,\mu_k,\tau\in\R$ are the Lagrange multipliers\footnote{The right-hand side is the projection of the unconstrained gradient of~$J_\rho$ on the normal space at~$(u,\psi)\in\Neh_\rho$, hence it is well-defined in the Hilbert space~$H^1\times H^{\frac{1}{2}}$. In particular, the series on the right-hand side converges. The same remark applies also in the sequel.}.
In the equation for~$u$ the term~$\delta_u\varphi_k(u)$ denotes the variation of~$\varphi_k(u)$ with respect to~$u$, which exists because of the analytic dependence of~$e^{-u}\D$ on~$u$, and
\begin{align}
 \D\delta_u\varphi_k
 =&\delta_u(\D\varphi_k)
    =\delta_u(\lambda_k(u)e^u\varphi_k(u)) \\
 =&\parenthesis{\delta_u\lambda_k(u)}e^u\varphi_k(u)
   +\lambda_k(u)e^u\varphi_k(u)
   +\lambda_k(u)e^u\delta_u\varphi_k(u) \in L^2
\end{align}
hence~$\|\delta_u\varphi_k(u)\|_{H^1}\le C(1+|\lambda_k(u)|)(1+\|\varphi_k(u)\|_{L^2})$.
\begin{lemma}\label{lemma:naturality}
 If~$(u,\psi)$ is a constrained critical point of~$J_\rho|_{\Neh_\rho}$, then it is also an unconstrained critical point of~$J_\rho$.
\end{lemma}
    \begin{proof}
     Suppose~$(u,\psi)\in N_\rho$ satisfies the constrained equations~\eqref{eq:constrained EL-u}-\eqref{eq:constrained EL-psi}: we need to show that all the Lagrange multipliers  vanish.

     First test~\eqref{eq:constrained EL-psi} against~$\varphi_k(u)$.
     By our choices this leads to~$\mu_k=0$, for any~$k<0$.
     Then testing~\eqref{eq:constrained EL-u} against the constant function~$1$, noting that~$G_1(u)=0$ and~$G_3(u,\psi)=0$, we get
     \begin{equation}
      2\tau\int_{\sph^2} 2e^{2u}+\rho e^u|\psi|^2\dv=0,
     \end{equation}
     and hence~$\tau=0$.
     It remains to show that if the system
     \begin{equation}\label{eq:alpha-EL-u}
      -\Delta u+1-e^{2u}-\rho e^u|\psi|^2
        =\sum_{j=1}^3 \alpha_j x^j e^{2u},
     \end{equation}
     \begin{equation}\label{eq:alpha-EL-psi}
      \D\psi-\rho e^u\psi=0,
     \end{equation}
     admits a solution, then~$\alpha_j=0$ for~$j=1,2,3$.

     Recall the \emph{basic identity} from~\cite{kazdan1975scalar}: given a Riemannian manifold~$(M,g)$
     and      two functions~$u, F\in C^\infty(M)$, it holds that
    \begin{equation}\label{eq:basic identity}
    2\Delta u (\nabla F\cdot \nabla u)=\diverg\left(2(\nabla F\cdot \nabla u)\nabla u-|\nabla u|^2\nabla F\right)-2(\Hess(F)-(\Delta F)g)(\nabla u,\nabla u).
    \end{equation}
    We will use this formula for~$F=x^j$ ($j=1,2,3$), which are the eigenfunctions of~$-\Delta_{\sph^2}$ associated to the first eigenvalue~$\mu_1(-\Delta_{\sph^2})=2$:
    $$
    -\Delta_{\sph^2} F=2F, \quad 2\Hess(F)-(\Delta_{\sph^2}F)g=0.
    $$
    Substituting into~\eqref{eq:basic identity} and then integrating over~$M=\sph^2$, we get
    \begin{equation}
    \int_{\sph^2} \Delta u(\nabla F\cdot \nabla u)\dv=0 .
    \end{equation}
    In our situation, following the notational convention  in~\cite{kazdan1975scalar},
    \begin{align}
        \Delta u
        = & \,1-\rho e^u|\psi|^2
        -\left(1-\sum_{j=1}^3\alpha_j x^j\right)e^{2u} \\
            \equiv & \,c-he^u-f e^{2u},
    \end{align}
    where we set~$c=1$, ~$h\equiv\rho |\psi|^2$ and~$f\equiv 1-\sum_{j=1}^3\alpha_j x^j$.
    Thus
    \begin{equation}
    c\int_{\sph^2}\nabla F\cdot \nabla u\dv
    =\int_{\sph^2}h e^u\nabla u\cdot\nabla F\dv
    + \int_{\sph^2}f e^{2u}\nabla F\cdot\nabla u\dv.
    \end{equation}
    Next we apply the argument in~\cite{kazdan1975scalar} to get
    \begin{align}
    LHS=&{}-c\int_{\sph^2}F(\Delta u)\dv
        =-c\int_{\sph^2}F(c-he^u -fe^{2u})\dv \\
     =&{}c\int_{\sph^2}h e^u F\dv+ c\int_{\sph^2}f e^{2u}F \dv_g,
    \end{align}
    where we  used the fact that~$\int_{\sph^2}F\dv=-\frac{1}{2}\int_{\sph^2}\Delta F \dv=0$; meanwhile
    \begin{align}
    RHS=&{}\int_{\sph^2}h \nabla(e^u)\cdot\nabla F\dv
        +\int_{\sph^2}f\nabla(e^{2u})\cdot\nabla F\dv \\
     =&{}-\int_{\sph^2}e^u\diverg(h\nabla F)\dv
        -\frac{1}{2}\int_{\sph^2}e^{2u}\diverg(f\nabla F)\dv \\
     =&{}-\int_{\sph^2}e^u(\nabla h\cdot\nabla F+h\Delta F)\dv
        -\frac{1}{2}\int_{\sph^2}e^{2u}(\nabla f\cdot\nabla F+ f\Delta F)\dv \\
     =&{}-\int_{\sph^2}e^u \nabla h\cdot\nabla F\dv
         +2\int_{\sph^2} e^u h F\dv
         -\frac{1}{2}\int_{\sph^2}e^{2u}\nabla f\cdot\nabla F\dv
         +\int_{\sph^2}e^{2u}f F\dv.
    \end{align}
    We thus get
    \begin{align}
    (2-c)\int_{\sph^2}e^u hF\dv+ (1-c)\int_{\sph^2}e^{2u}f F\dv
    =\int_{\sph^2} e^u\nabla h\cdot\nabla F\dv
        +\frac{1}{2}\int_{\sph^2}e^{2u}\nabla f\cdot\nabla F\dv.
    \end{align}
    That is, for each~$j=1,2,3$
    \begin{equation}
        \rho\int_{\sph^2}e^u|\psi|^2 x^j\dv
        =\rho\int_{\sph^2}e^u\nabla(|\psi|^2)\cdot\nabla x^j\dv
            +\frac{1}{2}\int_{\sph^2}e^{2u}\nabla(1-\sum_{i=1}^3 \alpha_i x^i)\cdot \nabla x^j\dv.
    \end{equation}
    Combining with the following Lemma~\ref{lemma:conservation law on mixed term}, we obtain that for each~$j$
    \begin{equation}
        \int_{\sph^2}e^{2u}\nabla(\sum_{i=1}^3 \alpha_i x^i)\cdot\nabla x^j\dv=0.
    \end{equation}
    Multiplying by~$\alpha_j$ and then summing over~$j$, we obtain
    \begin{equation}
        \int_{\sph^2} e^{2u}\left|\sum_{j=1}^3 \alpha_j \nabla x^j\right|^2 \dv=0.
    \end{equation}
    It follows that ~$\sum_{j=1}^3 \alpha_j \nabla x^j= 0$ everywhere on~$\sph^2$ and hence~$\alpha_j=0$ for each~$j=1,2,3$.
    Therefore~$(u,\psi)$ satisfies the unconstrained Euler--Lagrange equation~\eqref{eq:EL-uniformized metric}.
    \end{proof}

The following lemma describes a conservation law originating from the conformal invariance of the spinorial part of the functional~$J_\rho$.
This can be viewed as a generalization of some results in~\cite{kazdan1975scalar,changyang1987prescribing}.
\begin{lemma}\label{lemma:conservation law on mixed term}
 Let~$\psi\in\Gamma(S)$ be a spinor satisfying~$\D\psi-\rho e^u\psi=0$.
 Then for each~$j=1,2,3$, there holds
 \begin{equation}
  \int_{\sph^2}e^u\nabla(|\psi|^2)\cdot\nabla x^j\dv
  = \int_{\sph^2} e^u |\psi|^2 x^j\dv.
 \end{equation}
\end{lemma}

 \begin{proof}
  We prove the result for~$j=3$, the others cases being similar.

  Let~$\varphi_t\in \PSL(2;\C)$ be a smooth family of M\"obius transformations such that~$\varphi_0=\Id$ and
  \begin{equation}
   \frac{\dd}{\dd t}\Big|_{t=0} \varphi_t=\nabla x^3.
  \end{equation}
  Such a family can be obtained e.g. by pulling the dilation~$z\mapsto tz$ on~$\C$ back to the Riemann sphere~$\sph^2$ via the standard stereographic projection, see~\cite[Section 2]{changyang1987prescribing}.
  These are conformal diffeomorphisms:~$(\varphi_t^* g_0)_x=\det(\dd\varphi_t(x)){g_0}_x$.
  Let~$\beta\colon S\to S$ be the induced isometric isomorphism of the spinor bundle over~$\sph^2$.
  Define the family of spinors~$\psi_t\coloneqq \det(\dd\varphi_t)^{\frac{1}{4}}\beta(\psi\circ\varphi_t)\in\Gamma(S)$, ~$\psi_0=\psi$.
  Note that the Dirac action is preserved
  \begin{equation}
   \int_{\sph^2} \left<\D\psi,\psi\right>\dv
   =\int_{\sph^2} \left<\D\psi_t,\psi_t\right>\dv.
  \end{equation}
  Consider now the part in the functional containing spinors, i.e.
  \begin{equation}
   \int_{\sph^2} (\left<\D\psi,\psi\right>-\rho e^u|\psi|^2)\dv_g.
  \end{equation}
  Along the above smooth variation we have on one hand, by hypothesis,
  \begin{align}
   \frac{\dd}{\dd t}\Big|_{t=0}  \int_{\sph^2} ( \left<\D\psi_t,\psi_t\right>-\rho e^u|\psi_t|^2)\dv_g
   =2\int_{\sph^2}\left<\D\psi-\rho e^u\psi, \frac{\dd}{\dd t}\Big|_{t=0}\psi_t\right>\dv_g =0;
  \end{align}
  on the other hand, since the Dirac action part is already invariant, it follows that
  \begin{align}
   \frac{\dd}{\dd t}\Big|_{t=0}  \int_{\sph^2} (\left<\D\psi_t,\psi_t\right> &-\rho e^u|\psi_t|^2)\dv_g
   =-\rho\int_{\sph^2} e^u\frac{\dd}{\dd t}\Big|_{t=0}|\psi_t|^2\dv_g  \\
   =&-\rho\int_{\sph^2} e^u\frac{\dd}{\dd t}\Big|_{t=0}\left(\det(\dd\varphi_t)^{\frac{1}{2}}(|\psi|^2\circ\varphi_t)\right) \dv\\
   =&-\rho\int_{\sph^2} e^u\left(\frac{\dd}{\dd t}\Big|_{t=0}\det(\dd\varphi_t)^{\frac{1}{2}}\right)|\psi|^2\dv
    -\rho\int_{\sph^2} e^u \nabla(|\psi|^2)\cdot\left(\frac{\dd}{\dd t}\Big|_{t=0}\varphi_t\right) \dv \\
   =&-\rho\int_{\sph^2}e^u\left(\frac{1}{2}\Delta x^3\right)|\psi|^2\dv
   -\rho\int_{\sph^2}e^u\nabla(|\psi|^2)\cdot\nabla x^3 \dv \\
   =&\rho\int_{\sph^2}e^u|\psi|^2x^3\dv
   -\rho\int_{\sph^2}e^u\nabla(|\psi|^2)\cdot\nabla x^3 \dv .
  \end{align}
  In the last two steps we  used the fact that
  \begin{align}
   \frac{\dd}{\dd t}\Big|_{t=0}\det(\dd\varphi_t)^{\frac{1}{2}}
   =\frac{1}{2}\Delta x^3 = -x^3,
  \end{align}
  which can be checked by an elementary calculation, see Appendix~\ref{sect:appendix}.
  The desired conclusion follows.
 \end{proof}

We proved therefore that~$\Neh_\rho$ is a \emph{Nehari-type manifold}.
In the rest of the paper,  we will look for critical points of~$J_\rho|_{\Neh_\rho}$.

%%%%%%%%%%%%%%%%%%%
\section{Convergence of bounded Palais-Smale sequences}\label{sect:bounded PS seq subconverge}
In the last sections we will need to deal with bounded Palais-Smale sequences on the Nehari manifolds.
Here we first show that any bounded~$(PS)_c$ sequence (i.e. Palais-Smale sequence at level $c$) admits a strongly convergent sub-sequence.
We remark that, though we will not strictly use the result in this form, later on we will crucially rely on its proof.

Let~$(\rho_n)$ be a converging sequence with limit~$\rho_\infty\ge1$, and let~$c\in\R$.
Let~$(u_n,\psi_n)\in \Neh_{\rho_n}$ be a sequence such that
$$
 J_{\rho_n}(u_n,\psi_n)\to c, \quad
 \nabla^{\Neh_{\rho_n}} J_{\rho_n}(u_n,\psi_n)\to 0,  \qquad
 \textnormal{ as } n\to\infty.
$$
More precisely, for each~$n\ge 1$ there exist an affine function~$\alpha_n=\sum_{j}\alpha_{n,j}x^j\in(\R^3)^*\subset H^1(\sph^2)$, an auxiliary spinor~$\phi_n=\sum_{k<0}\mu_{n,k}\varphi_k(u_n)\in H^{\frac{1}{2}}(S)$, and a number~$\tau_n\in \R$ such that
\begin{align}\label{eq:PS-u}
 -\Delta u_n+1&-e^{2u_n}-\rho_n e^{u_n}|\psi_n|^2
 -\alpha_ne^{2u_n}\\
 &-2\Abracket{\D\psi_n-\rho_n e^{u_n}\psi_n,\delta_u\phi_n}
 +2\rho_n e^{u_n}\Abracket{\psi_n,\phi_n}
 -2\tau_n\parenthesis{2e^{2u_n}+\rho_n e^{u_n}|\psi_n|^2}
 =a_n,
\end{align}
\begin{align}\label{eq:PS-psi}
 \D\psi_n-\rho_n e^{u_n}\psi_n
 -\parenthesis{\D\phi_n-\rho e^{u_n}\phi_n}-\tau_n \rho_n e^{u_n}\psi_n
 =b_n,
\end{align}
with~$a_n\to 0$ in~$H^{-1}(\sph^2)$ and~$b_n\to 0$ in~$H^{-\frac{1}{2}}(S)$.
Here~$\delta_u\phi_n\equiv \sum_{k<0}\mu_{n,k}\delta_u\varphi_k(u) \in H^{\frac{1}{2}}(S)$.
Moreover, we assume that~$(u_n,\psi_n)_n$ are bounded in~$H^1(\sph^2)\times H^{\frac{1}{2}}(S)$.
By passing to a subsequence, we may assume that~$(u_n,\psi_n)$ converges weakly to a limit~$(u_\infty,\psi_\infty)\in H^1(\sph^2)\times H^{\frac{1}{2}}(S)$.

\begin{lemma}\label{lemma:Lagrange multiplier tends to zero}
 Let $(u_n, \psi_n)$, $\alpha_n, \tau_n$ and $\phi_n$ be as above. Then, by passing to a further subsequence, we have
 \begin{enumerate}
  \item ~$\phi_n\to 0$ in~$H^{\frac{1}{2}}(S)$;
  \item $\tau_n\to 0$ in~$\R$;
  \item $\alpha_n\to 0$ in~$\R^3$.
 \end{enumerate}
\end{lemma}
Thus the Lagrange multipliers are all tending to zero in the limit $n \to + \infty$.
    \begin{proof}[Proof of Lemma~\ref{lemma:Lagrange multiplier tends to zero}]
    We test~\eqref{eq:PS-psi} against~$\phi_n$ to get that
    \begin{equation}
     -\int_{\sph^2}\Abracket{\D\phi_n,\phi_n}\dv_g
     +\rho_n\int_{\sph^2} e^{u_n}|\phi_n|^2\dv
     =\Abracket{b_n,\phi_n}_{H^{-\frac{1}{2}}\times H^{\frac{1}{2}}}\le o(1)\|\phi_n\|_{H^{\frac{1}{2}}},
    \end{equation}
    which implies
    $$
     \|\phi_n\|_{H^{\frac{1}{2}}}\to 0, \quad \int_{\sph^2}e^{u_n}|\phi_n|\dv\to 0.
    $$
    This is equivalent to say that~$(|\lambda_k|^{\frac{1}{2}}\mu_k)_{k<0}\to 0$ in~$\ell^2$, and hence also~$\delta_u\phi_n\to 0$ in~$H^{\frac{1}{2}}(S)$.

    Thus, testing~\eqref{eq:PS-u} against the constant function~$1$ we obtain
    \begin{equation}
     -2\tau_n\int_{\sph^2} (2e^{u_n}+\rho_n e^{u_n}|\psi_n|^2)\dv
     =\Abracket{a_n,1}_{H^{-1}\times H^1}
     +2\int_{\sph^2}\Abracket{\D\psi_n-\rho_n e^{u_n}\psi_n, \delta_u\phi_n}\dv.
    \end{equation}
    Since the~$(u_n,\psi_n)$'s are assumed to be uniformly bounded and the above right-hand side  converges to zero, we conclude that~$\tau_n\to 0$ as~$n\to\infty$.

    Finally, testing~\eqref{eq:PS-u} against~$\alpha_n$ and using that the matrix~$\Lambda(u)$ has eigenvalues bounded both from above and below, we see that the~$\alpha_n$'s are uniformly bounded in~$(\R^3)^*$.
    Therefore, we may assume that~$(\alpha_n)$ converges weakly to~$\alpha_\infty\in (\R^3)^*$.
    By Sobolev's embedding theorems, we see that the weak limit~$(u_\infty,\psi_\infty)$ of the sequence~$(u_n,\psi_n)$ now satisfies the equations
    \begin{align}
     -\Delta u_\infty+1-e^{2u_\infty}-\rho_\infty e^{u_\infty}|\psi_\infty|^2
     &=\alpha_\infty e^{2u_\infty}, \\
     \D\psi_\infty-\rho_\infty e^{u_\infty}\psi_\infty
     &=0,
    \end{align}
    in~$H^{-1}(\sph^2)\times H^{-\frac{1}{2}}(S)$.
    Elliptic regularity theory implies that~$(u_\infty,\psi_\infty)$ is smooth and the argument to prove that~$\Neh_\rho$ is a natural constraint can be employed to show that~$\alpha_\infty=0$.
    It suffices to note that, since~$(\R^3)^*$ is finite-dimensional, the weak convergence coincides with the strong convergence, hence~$\alpha_n\to 0$ in~$(\R^3)^*$.
    \end{proof}

\begin{lemma}
 With the same notation as above,~$(u_n,\psi_n)$ converges to~$(u_\infty,\psi_\infty)$ strongly in~$H^1(\sph^2)\times H^{\frac{1}{2}}(S)$.
\end{lemma}
    \begin{proof}
    Since~$(u_n,\psi_n)$ converges to~$(u_\infty,\psi_\infty)$ weakly in~$H^1(\sph^2)\times H^{\frac{1}{2}}(S)$, we have
    \begin{align}
        e^{t u_n}&\to e^{t u_{\infty}}\quad\textnormal{ in } L^p(\sph^2), \quad \forall t\in\R, \forall p\in [1,\infty),
        \\
        \psi_n&\to\psi_\infty,\quad \textnormal{ in } L^q(S)\quad \forall q\in [1,4).
    \end{align}
    Now set~$\widetilde{u}_n\coloneqq u_n-u_\infty$ and~$\widetilde{\psi}_n\coloneqq \psi_n-\psi_\infty$.
    The difference of the spinors satisfies the equation
    \begin{align}
     \D\widetilde{\psi}_n
     =&\D\psi_n-\D\psi_\infty \\
     =& \parenthesis{\rho_n e^{u_n}\psi_n-\rho_\infty e^{u_\infty}\psi_\infty}
       +\parenthesis{\D\phi_n-\rho e^{u_n}\phi_n}
       +\tau_n\rho_n e^{u_n}+b_n \\
     =&\rho_n e^{u_n}(\psi_n-\psi_\infty)
     +\rho_n(e^{u_n}-e^{u_\infty})\psi_\infty
     +(\rho_n-\rho_\infty)e^{u_\infty}\psi_\infty \\
      &+\parenthesis{\D\phi_n-\rho e^{u_n}\phi_n}
       +\tau_n\rho_n e^{u_n}+b_n  \quad \to 0 \qquad \textnormal{ in } H^{-\frac{1}{2}}(S).
    \end{align}
    Since~$\D$ has no kernel, we see that~$\|\widetilde{\psi}_n\|_{H^{\frac{1}{2}}}\to 0$, that is~$\psi_n\to\psi_\infty$ in~$H^{\frac{1}{2}}(s)$.

    The same strategy works for the scalar components.
    Indeed,
    \begin{align}
     -\Delta\widetilde{u}_n
     =& -\Delta u_n+\Delta u_\infty \\
     =&\parenthesis{e^{2u_n}-e^{2u_\infty}}
      +\parenthesis{\rho_n e^{u_n}|\psi_n|^2
                    -\rho_\infty e^{u_\infty}|\psi_\infty|^2}
      +\alpha_n e^{2u_n} \\
      &+2\Abracket{\D\psi_n-\rho_n e^{u_n}\psi_n,\delta_u\phi_n}
       -2\rho_n e^{u_n}\Abracket{\psi_n,\phi_n}
       +2\tau_n(e^{2u_n}+\rho_n e^{u_n}|\psi_n|^2)+a_n.
    \end{align}
    Noting that
    \begin{align}
     \rho_n e^{u_n}|\psi_n|^2
                    -\rho_\infty e^{u_\infty}|\psi_\infty|^2
     =&\rho_n e^{u_n}\parenthesis{|\psi_n|^2-|\psi_\infty|^2}
        +\rho_n(e^{u_n}-e^{u_\infty})|\psi_\infty|^2
        +(\rho_n-\rho_\infty)e^{u_\infty}|\psi_\infty|^2,
    \end{align}
    which converges to zero in~$L^{\frac{4}{3}}(\sph^2)$, we see that~$-\Delta \widetilde{u}_n\to 0$ in~$H^{-1}(\sph^2)$.
    Since~$\|\widetilde{u}_n\|_{L^2}\to 0$, we conclude that~$\|\widetilde{u}_n\|_{H^1(\sph^2)}\to 0$, as desired.
    \end{proof}

\begin{rmk}
 Indeed one can show that any~$(PS)_c$ sequence is bounded.
 Combining this with the above result, we see that the functional~$J_\rho|_{\Neh_\rho}$ satisfies the Palais-Smale conditions.
\end{rmk}

%%%%%%%%%%%%%%%%%%%%%%%%
\section{Local geometry around the origin}
\label{sect:local critial groups}

We have seen that~$\Neh_\rho$ is a ~\emph{Nehari manifold} for the functional~$J_\rho$ and to prove existence of solutions to \eqref{eq:EL for super-Liouville} it suffices to find critical points of the restricted functional~$J_\rho|_{\Neh_\rho}$.
We first take a closer look at the local behavior of the functional around the  \emph{trivial} critical point~$\theta=(0,0)\in \Neh_\rho$, and then compute the critical groups at the origin.
Note that~$J_\rho(0,0)=0$.

The tangent space of~$\Neh_\rho$ at~$\theta$ is
\begin{align}
 T_\theta\Neh_\rho
 =&\braces{
  (v,h)\in H^1(\sph^2)\times H^{\frac{1}{2}}(S)\mid
  \int_{\sph^2}x^j v\dv=0 \;(j=1,2,3);\;
  \psi^-=0;\;
  \bar{v}=\fint_{\sph^2} v\dv =0
  } \\
  =&\Eigen(-\Delta_{\sph^2};\braces{0,2})^\perp\oplus H^{\frac{1}{2},+}(S).
\end{align}
Since~$\theta\in \Neh_\rho$ is a critical point, the local behavior of the functional~$J_\rho$ is determined by its Hessian, which is given by
\begin{align}
 \Hess(J_\rho|_{\Neh_\rho})[(v,h), (v,h)]
 =\int_{\sph^2} [2\left( |\nabla v|^2-2v^2 \right)
 +4\left<\D h-\rho h,h\right>]\dv.
\end{align}
Consider the case~$\rho\notin\Spec(\D)$:
on the finite-dimensional subspace
\begin{align}
 (T_\theta \Neh_\rho)^-\equiv
 \oplus_{\lambda<\rho}\Eigen(\D; \lambda)
  & &
  (\textnormal{with }\; l(\rho)\coloneqq \dim_\R (T_\theta \Neh_\rho)^-<\infty)
\end{align}
the Hessian is negative-definite, while on the complement subspace~$(T_\theta N_\rho)^+$ the Hessian is positive-definite.
In particular the Hessian~$\Hess(J_\rho)$ at~$\theta$ is non-degenerate and thus~$\theta$ is an isolated critical point.

We can then define the critical groups as in~\cite{marinoprodi1968teoria, chang1993infinite} for the functional~$J_\rho$ on $\Neh_\rho$ at the isolated critical point~$\theta=(0,0)$ as follows.
Let~$G$ be a non-trivial abelian group.
Let~$J_\rho^{c}$ denote the sublevel~$\{J_\rho\le c\}\cap \Neh_\rho$.
The critical groups of~$J_\rho|_{\Neh_\rho}$ at~$\theta\in \Neh_\rho$ are defined by:
\begin{equation}
 C_k(J_\rho|_{\Neh_\rho},\theta)
 \coloneqq
 H_k(J_\rho^{0}\cap U, (J_\rho^{0}\backslash\{\theta\})\cap U;G),
\end{equation}
where~$U$ is a small neighborhood such that there are no critical points in~$(J_\rho^{0}\backslash\{\theta\})\cap U$, and the right-hand side stands for the singular homology groups with coefficients in~$G$.
These groups are well-defined and independent of the choice of~$U$, thanks to the excision property.

By the above computation of the Hessian of~$J_\rho|_{\Neh_\rho}$ at~$\theta$, we see that
\begin{equation}
 C_k(J_\rho|_{\Neh_\rho},\theta)=
 \begin{cases}
  G & k= l(\rho), \\
  0 & k\neq l(\rho).
 \end{cases}
\end{equation}

%%%%%%%%%%%%%%%%%%%%%%%%%%%%%
\section{Local deformation of sublevels around non-bifurcation points}

We say that~$\rho_*$ is a \emph{bifurcation point} of~\eqref{eq:EL-uniformized metric} if there exist a sequence of numbers~$(\rho_n)$ and a sequence of non-trivial critical points~$(u_n,\psi_n)$ of~$J_{\rho_n}$ such that
\begin{equation}
 (u_n,\psi_n;\rho_n)\to (0,0;\rho_*) \quad \textnormal{ in } H^1(\sph^2)\times H^1(S) \times\R.
\end{equation}
In other words,~$(0,0;\rho_*)$ is an accumulation point of the set of non-trivial solutions
\begin{equation}
 \braces{(u,\psi;\rho): (u,\psi)\in \Neh_\rho \backslash\braces{(0,0)}, \; \dd J_\rho(u,\psi)=0}.
\end{equation}

\begin{thm}\label{thm:deformation around non-bifurcation point}
 Assume that
 \begin{equation} \label{hp:non bifurcation}  \tag{*}
  \rho_*>1 \textnormal{ is not a bifurcation point of ~\eqref{eq:EL-uniformized metric}}.
 \end{equation}
 Then there exist a number~$\varepsilon_1>0$ and two relative open neighborhoods~$U_{\rho_*\pm \varepsilon_1}$ of~$\theta=(0,0)$ in the corresponding sublevels:
 \begin{equation}
  U_{\rho_*\pm \varepsilon_1} \subset
  \braces{(u,\psi)\in \Neh_{\rho_*\pm \varepsilon_1}: J_{\rho_*\pm\varepsilon_1}(u,\psi)\le 0}
 \end{equation}
 such that~$U_{\rho_*-\varepsilon_1}$ is homeomorphic to~$U_{\rho_*+\varepsilon_1}$.
\end{thm}
We will assume  hypothesis~\eqref{hp:non bifurcation} until the proof of the above theorem, and we will take~$\delta>0$ and~$\varepsilon>0$ sufficiently small such that there are no non-trivial critical points of~$J_\rho$ in the neighborhood~$B_\delta(\theta)\cap \Neh_\rho$ for any~$\rho \in[\rho_*-\varepsilon,\rho_*+\varepsilon]$.
Such neighborhoods can of course be shrunk later if necessary.

Before the proof we need to state several lemmas.
We remark again that this result may be viewed as a nonlinear version of Krasnosel'skii Theorem and the proof goes in the spirit of~\cite{marinoprodi1968teoria}.

Introduce the following vector fields on~$B_\delta(\theta)\cap\Neh_\rho$:
\begin{align}
 Y_j(u)= &\parenthesis{(1-\Delta)^{-1}(x^j e^{2u}),\; 0}, \qquad j=1,2,3;\\
 Z_k(u,\psi;\rho)=& \big((1-\Delta)^{-1}(\Abracket{\D\psi-\rho e^u\psi,\delta_u\varphi_k(u)}-\rho e^u\Abracket{\psi,\varphi_k(u)}), \\
 &\qquad\qquad\qquad\qquad  |\D|^{-1}(\D\varphi_k(u)-\rho e^u\varphi_k(u))\big), \qquad k<0; \\
 W(u,\psi;\rho) =&\parenthesis{(1-\Delta)^{-1}(2e^{2u}+\rho e^u|\psi|^2),\; |\D|^{-1}(2\rho e^u \psi)}.
\end{align}
Note that
\begin{align}
 (1-\Delta)^{-1} x^j&=\frac{1}{3} x^j,  \qquad j=1,2,3; \\
 |\D|^{-1} \parenthesis{e^u \varphi_k(u)} &= \frac{1}{\lambda_k(u)} \varphi_k(u), \qquad \forall k<0,
\end{align}
with~$\lambda_k(u)=\lambda_k(0)+o(1)$ for~$u$ small, according to the  analytic dependence of the eigenvalues on the parameter~$u$.

\begin{lemma}
  %Assume~\eqref{hp:non bifurcation}.
  There exist~$\varepsilon>0$ and~$\delta>0$ such that in the ball~$B_\delta(\theta)\subset H^1(\sph^2)\times H^{\frac{1}{2}}(S)$
 the above vector fields are linearly independent, for each~$\rho\in[\rho_*-\varepsilon, \rho_*+\varepsilon]$.
\end{lemma}
    \begin{proof}
     We can estimate the following   inner products as:
        \begin{align}
        \Abracket{Y_j,Y_i}
        =& \int_{\sph^2} (1-\Delta)^{-1}(x^j e^{2u})\cdot x^i e^{2u} \dv \\
        =&\int_{\sph^2} (1-\Delta)^{-1}(x^j+2u x^j+ o(u)) \cdot (x^i+ 2u x^i+ o(u))\dv \\
        =&\int_{\sph^2} \frac{1}{3}x^j \cdot x^i \dv+O(\|u\|) \\
        =&\frac{4\pi}{9}\delta_{ij}+ O(\|u\|);
        \end{align}
        \begin{align}
         \Abracket{Z_k,Z_l}
         =&\int_{\sph^2}(1-\Delta)^{-1}\parenthesis{\Abracket{\D\psi-\rho e^u\psi,\delta_u\varphi_k(u)}-\rho e^u\Abracket{\psi,\varphi_k(u)}} \\
          &\qquad \qquad \qquad \cdot(\Abracket{\D\psi-\rho e^u\psi,\delta_u\varphi_l(u)}-\rho e^u\Abracket{\psi,\varphi_k(u)})\dv\\
          &\qquad+\int_{\sph^2}\Abracket{|\D|^{-1}(\D\varphi_k(u)-\rho e^u\varphi_k(u)), (\D\varphi_l-\rho e^u\varphi_l(u))}\dv \\
         =&O(\|\psi\|^2)+\int_{\sph^2}\Abracket{\frac{\lambda_k(u)-\rho}{\lambda_k(u)}\varphi_k(u),\parenthesis{\lambda_l(u)-\rho}e^u\varphi_l(u)}\dv\\
         =&\frac{(\lambda_k(u)-\rho)(\lambda_l-\rho)}{\lambda_k(u)}\delta_{kl}+O(\|\psi\|^2);
        \end{align}
        \begin{align}
         \Abracket{W,W}
         =&\int_{\sph^2}(1-\Delta)^{-1}(2e^{2u}+ e^u|\psi|^2)\cdot (2e^{2u}+\rho e^u|\psi|^2)\dv \\
          &+\int_{\sph^2}\Abracket{|\D|^{-1}(2\rho e^u\psi), 2\rho e^u\psi}\dv \\
         =&4+O(\|u\|+\|\psi\|^2);
        \end{align}
        \begin{align}
         \Abracket{Y_j,Z_k}
         =&\int_{\sph^2}(1-\Delta)^{-1}(x^j e^{2u})\cdot \parenthesis{\Abracket{\D\psi-\rho e^u\psi,\delta_u\varphi_k(u)}-\rho e^u\Abracket{\psi,\varphi_k(u)}}\dv
         =O(\|\psi\|);
        \end{align}
        \begin{align}
         \Abracket{Y_j,W}
         =&\int_{\sph^2}(1-\Delta)^{-1}(x^j e^{2u})\parenthesis{-\rho e^u\Abracket{\psi,\varphi_k(u)}}\dv  =O(\|\psi\|);
        \end{align}
        \begin{align}
         \Abracket{Z_k,W}
         =&O(\|\psi\|).
        \end{align}
        As a consequence of the last formulas, for~$(u,\psi)\in B_\delta(\theta)$ with~$\delta$ small, the above vector fields are linearly independent.
    \end{proof}

Introduce the~$\rho$-independent functionals
\begin{align}
 J^1(u,\psi)&\coloneqq\int_{\sph^2} |\nabla u|^2+2u+1-e^{2u}+2\Abracket{\D\psi,\psi}\dv, \\
 J^2(u,\psi)&\coloneqq\int_{\sph^2} 2 e^u|\psi|^2\dv.
\end{align}
Then we have
\begin{align}
 J_\rho(u,\psi)=J^1(u,\psi)-\rho J^2(u,\psi), & &
 \dd J_\rho(u,\psi)[v,h]
 =\dd J^1(u,\psi)[v,h]-\rho \dd J^2(u,\psi)[v,h].
\end{align}
The unconstrained gradients of the~$J^i$'s are
\begin{align}
  \grad J^1(u,\psi)=&\parenthesis{2(1-\Delta)^{-1}(-\Delta u+1-e^{2u}), \; 4|\D|^{-1}(\D\psi)}, \\
  \grad J^2(u,\psi)=&\parenthesis{2(1-\Delta)^{-1}(e^u|\psi|^2),
  \; 4|\D|^{-1}(e^u\psi)}.
 \end{align}

To find a deformation of the sublevels, we will focus on the level sets~$\{J_\rho=0\}$ with~$\rho$ close to~$\rho_*$, as it is done in the classical Morse theory.
More precisely, for~$\delta>0$ and~$\varepsilon>0$ small, let
\begin{align}
 \Omega\coloneqq &\braces{(u,\psi;\rho): (u,\psi)\in \Neh_\rho, \psi\neq 0, J_\rho(u,\psi)=0}, \\
 \Omega_\varepsilon\coloneqq &\braces{(u,\psi;\rho)\in\Omega: \rho\in[\rho_*-\varepsilon, \rho_*+\varepsilon]},\\
 M_\varepsilon\coloneqq& P(\Omega_\varepsilon),
\end{align}
where~$P\colon H^1(\sph^2)\times H^{\frac{1}{2}}(S)\times \R\to H^1(\sph^2)\times H^{\frac{1}{2}}(S) $ is the projection onto the first two factors.
Note that~$(u,\psi)\in M_\varepsilon$ implies that for some \emph{unique}~$\rho\in [\rho_*-\varepsilon, \rho_*+\varepsilon]$ such that~$J_\rho(u,\psi)=0$ and~$\psi\neq 0$.
In this case we will use~$\rho(u,\psi)$ to denote the dependence whenever~$(u,\psi)\in M_\varepsilon\cap B_\delta(\theta)$. For $u\in H^1(\mathbb{S}^2)$ we will write $u=\widehat{u}+\bar u$, where $\bar u$ is its average.

\begin{lemma}\label{lemma:boundes on J2=r}
 Let~$(u,\psi)\in B_\delta(\theta)\cap M_\varepsilon$.
 Suppose~$J^2(u,\psi)=4\pi r>0$ with~$r\ll \varepsilon<\delta$.
 Then, there exists a constant~$C>0$ such that
 \begin{equation}
  \|\nabla \widehat{u}\|^2_{L^2}+ |\bar{u}|+\|\psi\|^2_{H^{\frac{1}{2}}}\le Cr.
 \end{equation}
\end{lemma}
    \begin{proof}
     By definition, there exists~$\rho\in[\rho_*-\varepsilon, \rho_*+\varepsilon]$ such that~$(u,\psi)\in \Neh_\rho$ and~$J_\rho(u,\psi)=0$.
     In particular by  assumption
     \begin{align}
      \fint_{\sph^2} e^{2u}\dv =1-\rho\fint_{\sph^2}e^u|\psi|^2\dv
      =1-\rho \, r.
     \end{align}
     By Jensen's inequality,
     \begin{equation}
      e^{2\bar{u}}\le\fint_{\sph^2} e^{2u}\dv =1-\rho \, r,
     \end{equation}
     thus~$\bar{u}\le 0$.
     On the other hand, since~$\CM(e^{2u})=0$, an improved Moser-Trudinger inequality in~\cite{guimoradifam2018sphere} implies
     \begin{align}
      1-\rho \, r=\fint_{\sph^2}e^{2u}\dv
       \le \exp\parenthesis{\frac{1}{2}\fint_{\sph^2}|\nabla u|^2\dv+2\bar{u}}.
     \end{align}
     The  condition~$J_\rho(u,\psi)=0$ implies
     \begin{align}\label{eq:bounds on J2=r:level condition}
      \fint_{\sph^2} |\nabla u|^2+\Abracket{\D\psi,\psi}\dv
      =\rho \, r-2\bar{u}.
     \end{align}
     Moreover, since~$G_{2,k}(u,\psi)=0$ for all~$k<0$, the Dirac part is non-negative, and hence~$\fint_{\sph^2}|\nabla u|^2\dv\le \rho \, r-2\bar{u}$.
     Hence
     \begin{align}
      1-\rho \, r \le \exp\parenthesis{\frac{1}{2}\rho \, r+\bar{u}},
     \end{align}
     which gives a lower bound on~$\bar{u}$:
     \begin{equation}
      \bar{u} \ge \ln(1-\rho \, r)-\frac{1}{2}\rho \, r\ge -2\rho \, r.
     \end{equation}
     Therefore~$|\bar{u}|\le 2\rho \, r$ so by~\eqref{eq:bounds on J2=r:level condition} we see that~$\fint_{\sph^2}|\nabla u|^2\dv\le 5\rho \, r$, and
     \begin{equation}\label{eq:bound for psi}
      \fint_{\sph^2} \Abracket{\D\psi,\psi}\dv \le 5\rho \, r.
     \end{equation}
     In terms of the weighted basis~$\{\varphi_j(u)\}$ introduced in Section~\ref{section:natural constraint}, we can write~$\psi=\sum_{j>0}a_j(u,\psi)\varphi_j(u)$ with~$a_j(u,\psi)\in \R$ being the coefficients of the expansion, and~$\parenthesis{a_j(u,\psi)|\lambda_j(u)|^{1/2}}_j\in \ell^2$.
     Then~\eqref{eq:bound for psi} implies
     \begin{equation}
      0\le\sum_{j>0} a_j(u,\psi)^2 \lambda_j(u)\le 20\pi \rho \, r.
     \end{equation}
     Since~$\|u\|_{H^1}^2\le C\rho \, r$ we have~$\lambda_j(u)$ close to~$\lambda_j(0)$. It follows that~$\|\psi\|^2_{H^{\frac{1}{2}}(S_g)}\le C\rho \, r$.
    \end{proof}

\begin{lemma}\label{lemma:grad J2 is linearly independent}
 For~$\delta$ and~$\varepsilon$ small, ~$\grad J^2$ is linearly independent of~$Y_j$'s, $Z_k$'s and~$W$ on~$M_\varepsilon\cap B_\delta(\theta)$.
\end{lemma}
    \begin{proof}
     Suppose that
     \begin{equation}
      \grad J^2=\sum_{j=1}^3\alpha_j Y_j+\sum_{k<0}\mu_kZ_k+\tau W,
     \end{equation}
     namely
     \begin{align}
      2e^u|\psi|^2
      =&\sum_{j=1,2,3}\alpha_j x^j e^{2u}+\sum_{k<0}\mu_k \parenthesis{\Abracket{\D\psi-\rho e^u\psi,\Delta_u\varphi_k(u)}-\rho e^u\Abracket{\psi,\varphi_k(u)} }+\tau\parenthesis{2e^{2u}+\rho e^u|\psi|^2}, \\
      4e^u\psi
      =&\sum_{k<0}\mu_k\parenthesis{\D\varphi_k(u)-\rho e^u\varphi_k(u)} +2\tau\rho e^u\psi.
     \end{align}
     Testing the equation for the spinor against~$\varphi_l(u)$, we see that~$\mu_l=0$ for each~$l<0$.
     Since~$\psi\neq 0$, we conclude from the spinor equation that~$\rho \, \tau=2$.
     Then testing the scalar component of the equation  against the constant function~$1$ we see that~$\tau=0$, a contradiction.
    \end{proof}

\begin{lemma}\label{lemma:grad J1 is linearly independent}
 Assume~\eqref{hp:non bifurcation} holds.
 For~$\delta$ and~$\varepsilon$ small,~$\grad J^1$ is linearly independent of \,$\grad J^2$ and of \,$Y_j$'s, $Z_k$'s and~$W$ on~$M_\varepsilon\cap B_\delta(\theta)$.
\end{lemma}
    \begin{proof}
     Suppose that
     \begin{equation}
      \grad J^1=\lambda\grad J^2+\sum_{j=1}^3\alpha_j Y_j+\sum_{k<0}\mu_kZ_k+\tau W,
     \end{equation}
     for some~$\lambda,\alpha_j,\mu_k,\tau\in\R$.
     Expressed in components, we have
     \begin{align}
      2(-\Delta u+1-e^{2u})
      =&2\lambda e^u|\psi|^2
       +\sum_{j=1,2,3}\alpha_j x^j e^{2u} \\
       &+\sum_{k<0}\mu_k \parenthesis{\Abracket{\D\psi-\rho e^u\psi,\Delta_u\varphi_k(u)}-\rho e^u\Abracket{\psi,\varphi_k(u)} }+\tau\parenthesis{2e^{2u}+\rho e^u|\psi|^2}, \\
       4\D\psi
       =&4\lambda e^u\psi
       +\sum_{k<0}\mu_k\parenthesis{\D\varphi_k(u)-\rho e^u\varphi_k(u)} +2\tau\rho e^u\psi.
     \end{align}
     Testing the spinor equation against~$\varphi_k(u)$ we find that~$\mu_k=0$ for all~$k<0$.
     Then testing the scalar equation against the constant function~$1$, noting that~$G_1(u)=0$ and~$G_3(u,\psi)=0$, we have
     \begin{equation}\label{eq:lambda-rho-tau}
      2\rho\int_{\sph^2} e^u|\psi|^2\dv
      =2\lambda\int_{\sph^2} e^u|\psi|^2\dv
       +\tau\int_{\sph^2}1+e^{2u}\dv.
     \end{equation}
     Since~$(u,\psi)\in B_\delta(\theta)$ and~$\delta$ is small, we conclude that~$\tau=0$.
     Then we are in a situation similar to that of Lemma~\ref{lemma:naturality}.
     The same argument via M\"obius invariance implies that~$\alpha_j=0$ for~$j=1,2,3$.
     Now since~$\psi\neq 0$,~\eqref{eq:lambda-rho-tau} implies~$\lambda=\rho$.
     Thus~$(u,\psi)\in\Neh_\rho$ is a non-trivial critical point of~$\Neh_\rho$, contradicting  hypothesis~\eqref{hp:non bifurcation}.
    \end{proof}

The vector fields~$Y_j$'s,~$Z_k$'s and~$W$ form a local frame for the normal bundle~$T^\perp \Neh_\rho$ on~$B_\delta(\theta)\cap \Neh_\rho$, which are almost orthogonal.
We denote the tangent parts of the gradients of~$J^i$,~($i=1,2$), by~$\nabla^{\Neh_{\rho}} J^i$.
Next we show that the latter constrained gradients are \emph{not} collinear in a uniform sense wherever~$J^2(u,\psi)$ is strictly away from zero.
The collinearity of the two constrained gradients is measured by the determinant of the following matrix
\begin{equation}
 \begin{pmatrix}
  \Abracket{\nabla^{\Neh_\rho}J^1(u,\psi),\nabla^{\Neh_\rho}J^1(u,\psi)} & \Abracket{\nabla^{\Neh_\rho}J^1(u,\psi),\nabla^{\Neh_\rho}J^2(u,\psi)} \vspace{0.2cm}\\
  \Abracket{\nabla^{\Neh_\rho}J^2(u,\psi),\nabla^{\Neh_\rho}J^1(u,\psi)}
  &\Abracket{\nabla^{\Neh_\rho}J^2(u,\psi),\nabla^{\Neh_\rho}J^2(u,\psi)}
 \end{pmatrix},
\end{equation}
which is
\begin{equation}
 \det(J^1, J^2)(u,\psi;\rho)\equiv
 \|\nabla^{\Neh_\rho}J^1(u,\psi)\|^2\|\nabla^{\Neh_\rho}J^2(u,\psi)\|^2
 -\Abracket{\nabla^{\Neh_\rho}J^1(u,\psi),\nabla^{\Neh_\rho}J^2(u,\psi)}^2,
\end{equation}
and it is non-negative by the Cauchy-Schwarz inequality.
Recall that in~$M_\varepsilon\cap B_\delta(\theta)$,~$\psi\neq0$ and hence~$J^2(u,\psi)\neq 0$ and~$\grad J^2(u,\psi)\neq 0$.
Thus we can write
\begin{align}
  \det(J^1,J^2)(u,\psi,\rho)
  =\|\nabla^{\Neh_\rho} J^2(u,\psi)\|^2
  \left\| \nabla^{\Neh_\rho} J^1(u,\psi)-\Abracket{\nabla^{\Neh_\rho} J^1,
  \frac{\nabla^{\Neh_\rho}J^2}{\|\nabla^{\Neh_\rho}J^2\|}
  }
  \frac{\nabla^{\Neh_\rho}J^2}{\|\nabla^{\Neh_\rho}J^2\|}(u,\psi)
  \right\|^2.
 \end{align}
We deal with the right-hand sides separately.
\begin{lemma}
 There exists a modulus of continuity~$\kappa\colon[0,\delta]\to[0,1]$ such that
 \begin{equation}
  \|\nabla^{\Neh_\rho}J^2(u,\psi)\|^2\ge \kappa(J^2(u,\psi)).
 \end{equation}
\end{lemma}
    \begin{proof}
     We first claim that for each~$0<r<\delta$,
        \begin{equation}
        \widetilde{\kappa}(r)\coloneqq
            \inf_{\substack{(u,\psi)\in M_\varepsilon\cap B_\delta(\theta)\\ J^2(u,\psi)=r}} \|\nabla^{\Neh_\rho}J^2(u,\psi)\|^2>0.
        \end{equation}
     Otherwise, there would exist~$\rho_n\in[\rho_*-\varepsilon,\rho_*+\varepsilon]$ and~$(u_n,\psi_n)\in \Neh_{\rho_n}$ with~$J_{\rho_n}(u_n,\psi_n)=0$ and
        $$
        J^2(u_n,\psi_n)=\int_{\sph^2} e^{u_n}|\psi_n|^2\dv=r, \quad
        \|\nabla^{N_{\rho_n}} J^2(u_n,\psi_n)\|\to 0.
       $$
    This means that there exist~$\alpha_{n,j}\in\R$ and~$\mu_{n,k}\in\R$ and~$\tau_n\in\R$ such that
    \begin{align}
     2e^{u_n}|\psi_n|^2-\sum_{j=1}^3 \alpha_{n,j}x^j e^{2u}-\sum_{k<0}\mu_{n,k}\big(\Abracket{\D\psi_n-\rho e^{u_n}\psi_n,\delta_u\varphi_k(u_n)}
     -&\rho_n e^{u_n}\Abracket{\psi_n,\varphi_k(u_n)}\big) \\
     -\tau_n(2 e^{2u_n}+\rho_n e^{u_n}|\psi_n|^2)
     =& a_n \to 0 \qquad \textnormal{ in } H^{-1}(\sph^2), \\
     4e^{u_n}\psi_n-\sum_{k<0} \mu_{n,k} (\D\varphi_k(u_n)-\rho_n e^{u_n}\varphi_k(u_n))-2\tau_n \rho_n e^{u_n}\psi_n
     =& b_n\to 0\qquad \textnormal{ in } H^{-\frac{1}{2}}(S).
    \end{align}
    Reasoning as in Section~\ref{sect:bounded PS seq subconverge} we can show that the Lagrange multipliers are uniformly bounded.
    By passing to a subsequence we can extract weakly convergent subsequences such that at the weak limit~$(u_\infty,\psi_\infty)$ the constrained gradient vanishes, i.e.~$\nabla^{\Neh_{\rho}}J^2(u_\infty,\psi_\infty)=0$.
    However,
    \begin{equation}
     \int_{\sph^2} e^{u_\infty}|\psi_\infty|^2\dv
     =\lim_{n\to\infty} \int_{\sph^2} e^{u_n}|\psi_n|^2\dv =r
    \end{equation}
    due to the compactness of the Rellich embedding and Moser-Trudinger embedding.
    This contradicts Lemma~\ref{lemma:grad J2 is linearly independent}.
    Thus the claim is confirmed.

    The function~$\widetilde{\kappa}(r)$ defined above might not be continuous and monotonically non-decreasing, but at each~$0<r<\delta$ we may always replace the value~$\widetilde{\kappa}(r)$ by a smaller one to obtain a continuous, monotonically non-decreasing function~$\kappa$, as desired.
    \end{proof}

It remains to deal with
\begin{equation}
 P^\perp_2(\nabla^{\Neh_\rho}J^1)(u,\psi)
 \equiv \nabla^{\Neh_\rho} J^1(u,\psi)-\Abracket{\nabla^{\Neh_\rho} J^1,
  \frac{\nabla^{\Neh_\rho}J^2}{\|\nabla^{\Neh_\rho}J^2\|}
  }
  \frac{\nabla^{\Neh_\rho}J^2}{\|\nabla^{\Neh_\rho}J^2\|}(u,\psi).
\end{equation}

\begin{lemma}
 There exists a modulus of continuity~$\sigma\colon [0,\delta]\to [0, 1]$ such that
 \begin{equation}
  \|P^\perp_2(\nabla^{\Neh_\rho}J^1)(u,\psi)\|^2\ge \sigma(J^2(u,\psi)), \qquad \forall (u,\psi)\in B_\delta(\theta)\cap M_\varepsilon.
 \end{equation}
\end{lemma}
    \begin{proof}
     The proof is similar to the previous one, so we omit the details.
     One can first show that the Lagrange multipliers are uniformly bounded as in Section~\ref{sect:bounded PS seq subconverge}, and then pass to weakly convergent subsequences:
     this time the weak limits contradict Lemma~\ref{lemma:grad J1 is linearly independent}.
    \end{proof}

Summing-up, we obtained that on~$M_\varepsilon\cap B_\delta(\theta)$,
\begin{equation}\label{eq:modulus of continuity of det}
 \det(J^1,J^2)(u,\psi;\rho)\ge
 \parenthesis{\kappa\cdot \sigma}\parenthesis{\int_{\sph^2}e^u|\psi|^2\dv}>0.
\end{equation}

\begin{lemma}
 Assume~\eqref{hp:non bifurcation} holds.
 Then, for~$\delta$ and~$\varepsilon$ small, there exists a~$C^1$-vector field~$X=(X_u,X_\psi)\in H^1(\sph^2)\times H^{\frac{1}{2}}(S)$ on~$M_\varepsilon\cap B_\delta(\theta)$ such that
 \begin{align}
  \Abracket{X,\grad J^1}=&\int_{\sph^2}e^u|\psi|^2\dv, \\
  \Abracket{X,\grad J^2}=&\,0, \\
  \Abracket{X,Y_j}=&\,0,\qquad\qquad \forall j=1,2,3,  \\
  \Abracket{X,Z_k}=&\int_{\sph^2} e^u\Abracket{\psi,\varphi_k(u)}\dv=0, \qquad\qquad \forall k<0, \\
  \Abracket{X,W}=&\int_{\sph^2}e^u|\psi|^2\dv.
 \end{align}
\end{lemma}
    \begin{proof}
     At each~$(u,\psi)\in M_\varepsilon$, we need to solve a linear system with the coefficient-matrix being non-degenerate, due to the above lemmas.
     Such a system can thus be \emph{uniquely} solved in the space
     \begin{equation}
      \Span_{\R} \braces{\grad J^1, \grad J^2, Y_j\textnormal{'s}, Z_k\textnormal{'s}, W}.
     \end{equation}
     Since the coefficients of these linear systems depend on~$(u,\psi)$ in the~$C^1$ sense, so does the solution~$X(u,\psi)$.
    \end{proof}

In the sequel we denote by~$X(u,\psi;\rho)$ the unique vector field from the above lemma, which has a decomposition
\begin{equation}
 X(u,\psi;\rho)=X^\top(u,\psi;\rho)+X^\bot(u,\psi;\rho)
 \in T_{(u,\psi)}\Neh_\rho \oplus T_{(u,\psi)}^\perp \Neh_{\rho}.
\end{equation}
Then, explicitly, at~$(u,\psi)\in \Neh_\rho$,
\begin{align}\label{eq:vector field X-tangent part}
 X^\top(u,\psi;\rho)
 =&J^2(u,\psi)\frac{\|\nabla^{\Neh_\rho}J^2(u,\psi;\rho)\|^2}{\det(J^1,J^2)(u,\psi;\rho)}P^\perp_2(\nabla^{\Neh_\rho}J^1(u,\psi))
\end{align}
and up to higher order terms,
\begin{align}\label{eq:vector field X-normal part}
 X^\bot(u,\psi;\rho)
 =\sum_{k<0}\frac{\int_{\sph^2}e^u\Abracket{\psi,\varphi_k}\dv}{\|Z_k\|^2}Z_k(u,\psi;\rho)
 + \frac{J^2(u,\psi)}{\|W\|^2}W
   + O(\|u\|+\|\psi\|).
\end{align}

\

Now let~$0<2\varepsilon_1<\varepsilon<\delta$, and take a cut-off function~$\eta\in C^\infty_c([\varepsilon,\varepsilon])$ such that~$\eta\equiv 1$ on~$[-2\varepsilon_1,2\varepsilon_1]$.
Then set
\begin{align}
 \omega(u,\psi)
 \coloneqq \eta\parenthesis{
  \rho_*-\frac{J^1(u,\psi)}{J^2(u,\psi)}
 }
 \cdot\eta \parenthesis{\|u\|^2+\|\psi\|^2}
 \qquad \mbox{ in }
 \{(u,\psi)\in B_\delta(\theta)\mid \psi\neq 0\}.
\end{align}
Observe that, if~$(u,\psi)$ with~$\psi\neq 0$ satisfies~$\omega(u,\psi)\neq 0$, then there exists a unique~$\rho\in [\rho_*-\varepsilon,\rho_*+\varepsilon]$ such that
\begin{equation}
 J_\rho(u,\psi)=J^1(u,\psi)-\rho J^2(u,\psi)=0,
\end{equation}
hence~$(u,\psi)\in B_\delta(\theta)\cap M_\varepsilon$.
We define a vector field~$\widetilde{X}(u,\psi)$ on~$B_\delta(\theta)$ by
\begin{equation}
 \widetilde{X}(u,\psi)=
 \begin{cases}
  \omega(u,\psi)X(u,\psi;\rho) & \textnormal{ if } \psi \neq 0, \\
  0 & \textnormal{ if } \psi=0.
 \end{cases}
\end{equation}
Consider the flow generated by
\begin{equation}\label{eq:flow}
 \frac{\dd}{\dd \rho} (u,\psi)=\widetilde{X}(u,\psi).
\end{equation}
More precisely, for each~$(u_0,\psi_0)\in B_\delta(\theta)$ there exist families of trajectories~$(u(\rho),\psi(\rho))_{\rho\in\R}$ satisfying the following properties:
\begin{itemize}
 \item If~$\psi_0=0$, then~$(u(\rho),\psi(\rho))\equiv (u_0,\psi_0)$ for any~$\rho\in\R$.
 \item If~$\psi_0\neq 0$ and~$\omega(u_0,\psi_0)=0$, then again~$(u(\rho),\psi(\rho))\equiv(u_0,\psi_0)$ for any~$\rho\in\R$.
 \item If~$\psi_0\neq 0$ and~$\omega(u_0,\psi_0)\neq 0$, then as observed above, there exists a unique~$\rho_0\in[\rho_*-\varepsilon, \rho_*+\varepsilon]$ such that~$J_{\rho_0}(u_0,\psi_0)=0$, then~$\parenthesis{u(\rho),\psi(\rho)}$ solves the ODE
 \begin{equation}
  \begin{cases}
   \frac{\dd}{\dd\rho}(u(\rho),\psi(\rho))= \widetilde{X}(u(\rho),\psi(\rho)), \vspace{0.2cm}\\
   (u,\psi)_{|_{\rho=\rho_0}}= (u_0,\psi_0).
  \end{cases}
 \end{equation}
 To see that the solution exists for all~$\rho\in\R$, it suffices to show that the vector field~$\widetilde{X}$ is of class~$C^1$, bounded along the trajectory and that any trajectory segment has closure contained in the domain~$B_\delta(\theta)$.
 Since
     \begin{align}
      \frac{\dd}{\dd\rho}J^2(u(\rho),\psi(\rho))
      =\Abracket{\grad J^2,\widetilde{X}}(u,\psi)
      =\omega(u,\psi)\Abracket{\grad J^2, X}\equiv 0,
     \end{align}
 it follows that~$J^2(u(\rho),\psi(\rho))=const.$ for any~$\rho\in \R$ wherever the flow is defined.
 By~\eqref{eq:modulus of continuity of det} we see that~$\widetilde{X}$ is of class~$C^1$ and bounded along the trajectory.
 Consider the trajectory segment~$\{(u(\rho),\psi(\rho)): \rho \in [\rho_0, b)\}$.
 Taking the limit~$\rho\nearrow b$, the limit point evidently lies inside~$B_\delta(\theta)$ since~$\widetilde{X}$ vanishes on~$B_\delta(\theta)\backslash B_\varepsilon(\theta)$.
 Hence, by~\cite[Lemma 1.1]{marinoprodi1968teoria} the flow exists globally.
 Note that in this case the flow never stops at finite time, hence~$(u(\rho),\psi(\rho))\in \supp(\omega)$ and so~$(u(\rho),\psi(\rho))\in B_\delta(\theta)\cap M_\varepsilon$.
\end{itemize}

We will focus on those flow lines passing through the Nehari manifolds.

\begin{lemma}
 Assume~\eqref{hp:non bifurcation} holds and use the above notation.
 \begin{enumerate}
  \item In  case~$\psi_0\neq 0$ and~$\omega(u_0,\psi_0)\neq 0$, if ~$(u_0,\psi_0)\in \Neh_{\rho_0}$ with~$\rho_0$ satisfying~$J_{\rho_0}(u_0,\psi_0)=0$, then the trajectory~$(u(\rho),\psi(\rho))$ stays inside the manifold~$N$, namely
        \begin{itemize}
         \item $G_1(u(\rho))=G_1(u_0)=0$,~$\forall \rho\in \R$;
         \item $G_{2,k}(u(\rho),\psi(\rho))=0$,~$\forall\rho \in\R$ and~$\forall k<0$.
        \end{itemize}
  \item In addition, if~$(u_0,\psi_0)\in B_{\varepsilon_1}(\theta)\cap M_{\varepsilon_1}$, then for any~$\rho\in [\rho_*-\varepsilon_1,\rho_*+\varepsilon_1]$ we have
      \begin{itemize}
       \item $G_3(u(\rho),\psi(\rho))=0$, and in particular~$(u(\rho),\psi(\rho))\in \Neh_\rho$;
       \item $J_{\rho} (u(\rho),\psi(\rho))=0$.
      \end{itemize}
 \end{enumerate}

\end{lemma}

    \begin{proof}
      (1)
      In this case~$(u_0,\psi_0)\in M_\varepsilon\cap B_\delta(\theta)$.
      For the conservation of~$G_1^i$ ($i=1,2,3$): we have~$G_1(u(\rho_0))=0$, and
     \begin{align}
      \frac{\dd}{\dd\rho} G_1^i(u(\rho))
      =\int_{\sph^2} x^ie^{2u(\rho)}\cdot 2\frac{\dd u}{\dd \rho}\dv
      = \omega(u(\rho),\psi(\rho))\Abracket{Y_i, X}\equiv 0.
     \end{align}
      Similarly, for each $k<0$:
     \begin{align}
      \frac{\dd}{\dd\rho} \int_{\sph^2}
      \Abracket{\D\psi-\rho e^u\psi, \varphi_k(u(\rho))}\dv
      =&\omega(u(\rho),\psi(\rho))\Abracket{X,Z_k}-\int_{\sph^2}e^{u(\rho)}\Abracket{\psi(\rho),\varphi_k(u(\rho))}\dv\\
      =&-\frac{1}{\lambda(u(\rho))-\rho}\int_{\sph^2}\Abracket{\D\psi-\rho e^{u(\rho)}\psi(\rho),\varphi_k(u(\rho))},
     \end{align}
     where we  used the fact that~$\Abracket{X, Z_k}=0$ for~$(u,\psi)\in M_\varepsilon\cap B_\delta(\theta)$.
     Thus
     \begin{equation}
      \frac{\dd}{\dd \rho} G_{2,k}(u(\rho),\psi(\rho))
      =-\frac{1}{\lambda_k(u(\rho))-\rho} G_{2,k}(u(\rho),\psi(\rho)).
     \end{equation}
     Since~$\Abracket{X,Z_k}(u(\rho_0),\psi(\rho_0))=0$ and~$G_{2,k}(u(\rho_0),\psi(\rho_0))=G_{2,k}(u_0,\psi_0)=0$, it follows that
     \begin{equation}
      G_{2,k}(u(\rho),\psi(\rho))\equiv0,\qquad \forall \rho\in\R.
     \end{equation}

     (2) If~$J_{\rho_0}(u_0,\psi_0)=0$ for some~$\rho_0\in [\rho_*-\varepsilon_1,\rho_*+\varepsilon_1]$ and~$(u,\psi)\in B_{\varepsilon_1}(\theta)$, then
     \begin{equation}
      \rho_0-\frac{J^1(u_0,\psi_0)}{J^2(u_0,\psi_0)}=0
     \end{equation}
     and~$\omega(u_0,\psi_0)=1$.
     Hence there is a relatively open neighborhood~$V$ of~$\rho_0$ such that for~$\rho \in V$, we have
     \begin{equation}
      \omega(u(\rho),\psi(\rho))=1,
     \end{equation}
     and as a consequence
     \begin{align}
      \frac{\dd}{\dd \rho} G_3(u(\rho),\psi(\rho))
      =& \Abracket{\widetilde{X},W}-\int_{\sph^2} e^{u(\rho)}|\psi(\rho)|^2\dv \\
      =& \Abracket{X,W}- J^2(u(\rho),\psi(\rho))=0,
     \end{align}
     \begin{align}
      \frac{\dd}{\dd \rho} J_{\rho}(u(\rho),\psi(\rho))
      =&\Abracket{\grad J^1,\widetilde{X}}-J^2(u(\rho),\psi(\rho))
       -(\rho) \Abracket{\grad J^2, \widetilde{X}} \\
      =&\Abracket{\grad J^1,X}- J^2(u(\rho),\psi(\rho)) =0.
     \end{align}
     Thus~$\{\rho\in [\rho_*-\varepsilon_1,\rho_*+\varepsilon_1]: J_{\rho}(u(t),\psi(t))=0, G_3(u(\rho),\psi(\rho))=0\}$ is both an open and closed subset of~$[\rho_*-\varepsilon_1,\rho_*+\varepsilon_1]$, hence it coincides with the whole interval.
    \end{proof}

Now we can use this flow to find a deformation of the local sublevel sets.

\begin{proof}[Proof of Theorem~\ref{thm:deformation around non-bifurcation point}]
 Under the hypothesis~\eqref{hp:non bifurcation}, choose~$\varepsilon_1$ as above.
 Define the map
 \begin{equation}
  \Phi\colon  (B_{\varepsilon_1}\cap M_{\varepsilon_1})\times\R\times\R \to
  B_{\delta}(\theta)
 \end{equation}
 by~$\Phi((u_0,\psi_0);\rho_1,\rho_2)\coloneqq (u(\rho_2),\psi(\rho_2))$, where~$(u(\rho),\psi(\rho))$ is the flow generated by~\eqref{eq:flow} with initial condition~$(u(\rho_1),\psi(\rho_1))=(u_0,\psi_0)$.

 We claim that the map~$\Phi$ is continuous.
 It is clearly continuous when~$\psi\neq 0$ by the continuous dependence on the initial data, thus it remains to show that when~$J^2(u,\psi)=4\pi r$ is small, the flow stays close (in the spinor component) to the subspace~$\{\psi=0\}$.
 Since the flow line stays inside the set~$B_{\varepsilon_1}(\theta)\cap M_{\varepsilon_1}$, Lemma~\ref{lemma:boundes on J2=r} guarantees that the set~$B_{\varepsilon_1}(\theta)\cap M_{\varepsilon_1}\cap \{J^2=4\pi r\}$ is close to the origin, hence the flow is globally continuous and so is the map~$\Phi$.

 Consider the set
 \begin{equation}
  U_{\rho_*-\varepsilon_1}\coloneqq B_{\varepsilon_1}(\theta)\cap J^0_{\rho_*-\varepsilon_1}:
 \end{equation}
 then~$\Phi((\cdot,\cdot);\rho_*-\varepsilon_1, \rho_*+\varepsilon_1)$ carries~$U_{\rho_*-\varepsilon_1}$ to a relative neighborhood~$U_{\rho_*+\varepsilon_1}$ of~$\theta$ in~$J^0_{\rho_*+\varepsilon_1}$, and the inverse map is given by~$\Phi((\cdot,\cdot);\rho_*+\varepsilon_1, \rho_*-\varepsilon_1)$.
\end{proof}

\begin{proof}[Proof of Theorem~\ref{thm:existence around eigenvalues}]
 Let~$\rho_*=\lambda_k\in\Spec(\D)$ for some~$\lambda_k>1$.
 If~$\rho_*$ is not a bifurcation point, then by Theorem~\ref{thm:deformation around non-bifurcation point}, there are relatively open local neighborhoods~$U_{\rho_*\pm\varepsilon_1}$ in the sublevel sets of~$J_{\rho_*\pm\varepsilon_1}$ respectively, which are homeomorphic to each other.
 Hence the local critical groups for~$J_{\rho_*\pm\varepsilon_1}$ at~$\theta$ should be isomorphic.

 However, in Section~\ref{sect:local critial groups} we have seen that the local critical groups at~$\theta$ for~$J_{\rho_*-\varepsilon}$ and~$J_{\rho_*+\varepsilon_1}$ are different, which gives a contradiction.
\end{proof}

%%%%%%%%%%%%%%%%%%%%%%%
\section{Appendix: a conformal transformation}
\appendix
\label{sect:appendix}

In this appendix we perform for the reader's convenience the explicit computation used in Section~\ref{section:natural constraint}.

Consider the conformal transformation~$\varphi_t\colon \sph^2\to \sph^2$ defined by the following formulas
\begin{center}
 \begin{tikzcd}
  \sph^2\arrow[r,"\varphi_t"]\arrow[d,"\pi"] &\sph^2\arrow[d,"\pi"]
            & & \vec{x}=(x^1,x^2,x^3)\arrow[r,"\varphi_t"]\arrow[d,"\pi"]
                    & \vec{y}=(y^1,y^2,y^3)   \\
  \C\arrow[r,"t"] & \C
                & & z=\frac{x^1+ix^2}{1-x^3}\arrow[r, "t"]
                    &w=tz=\frac{(tx^1)+i(tx^2)}{1-x^3}\arrow[u,"\pi^{-1}"]
 \end{tikzcd}
\end{center}
where~$\pi\colon \sph^2\to \C$ denotes the stereographic projection.
Let us compute the curve~$\vec{y}=\vec{y}(t)$.
Note that
\begin{equation}
 |w|^2=\frac{(x^1)^2+(x^2)^2}{(1-x^3)^2} t^2,
\end{equation}
thus
\begin{align}
 y^1(t)=&\frac{2}{1+|w|^2}Re(w)
    =\frac{2tx^1(1-x^3)}{t^2((x^1)^2+(x^2)^2)+(1-x^3)^2}, \\
 y^2(t)=&\frac{2}{1+|w|^2}Im(w)
    =\frac{2tx^2(1-x^3)}{t^2((x^1)^2+(x^2)^2)+(1-x^3)^2}, \\
 y^3(t)=&\frac{|w|^2-1}{|w|^2+1}
    =\frac{t^2((x^1)^2+(x^2)^2)-(1-x^3)^2}{t^2((x^1)^2+(x^2)^2)+(1-x^3)^2}.
\end{align}
The~$t$-derivatives are
\begin{align}
 \frac{\dd}{\dd t}y^1(t)
 =&\frac{2x^1(1-x^3)}{[t^2((x^1)^2+(x^2)^2)+(1-x^3)^2]^2}
    \left[(1-x^3)^2-t^2((x^1)^2+(x^2)^2)\right]
    =-\frac{1}{t}y^1 y^3, \\
 \frac{\dd}{\dd t}y^2(t)
 =&\frac{2x^2(1-x^3)}{[t^2((x^1)^2+(x^2)^2)+(1-x^3)^2]^2}
    \left[(1-x^3)^2-t^2((x^1)^2+(x^2)^2)\right]
    =-\frac{1}{t}y^2 y^3, \\
 \frac{\dd}{\dd t}y^3(t)
 =&\frac{2x^1(1-x^3)}{[t^2((x^1)^2+(x^2)^2)+(1-x^3)^2]^2}
    \cdot 2(1-x^3)^2
    =\frac{1}{t}(1- (y^3)^2). \\
\end{align}
Note that the gradient of the coordinate function~${x}^3$ (at the point~$\vec{x}\in \sph^2$) is given by
\begin{equation}
 \grad{x}^3 (\vec{x})
 =\frac{\p}{\p x^3}
  -\left<\frac{\p}{\p x^3}, \vec{x}\right>_{\R^3}\vec{x}
 =(0,0,1)-x^3(x^1,x^2,x^3)
 =\left(-x^1 x^3, -x^2 x^3, 1-(x^3)^2\right).
\end{equation}
Since~$\vec{y}|_{t=1}=\vec{x}$, we have
\begin{equation}
  \frac{\dd}{\dd t}\Big|_{t=1}\varphi_{t}(\vec{x})
 =\frac{\dd}{\dd t}\Big|_{t=1} \vec{y}(t)
 = \grad(x^3).
\end{equation}
Moreover,
\begin{align}
 \frac{\dd}{\dd t}\Big|_{t=1}\det(\dd\varphi_t)
 =& \tr\left((\dd\varphi_t)^{-1}\frac{\dd}{\dd t}(\dd\varphi_t)\right)_{\big|t=1}
 =\tr\dd\left(\frac{\dd}{\dd t}\varphi_t\right)_{\big|t=1} \\
 =&\diverg(\grad (x^3))
 =\Delta_{\sph^2} x^3.
\end{align}

\end{document}